\documentclass{amsart}

\author{Andy Hammerlindl}
\title{Partial hyperbolicity on 3-dimensional nilmanifolds}

\usepackage{amsmath}
\usepackage{amssymb}
\usepackage{amsfonts}
\usepackage{amsthm}
\usepackage{amscd}
\usepackage{graphicx}
\usepackage{subfig}
\usepackage{hyphenat}

\newcommand{\defmath}[2]{\def#1{{\ensuremath{#2}%
}%
}}

\newcommand{\smallinv}{^{-1}\hspace{-0.13in}}

\defmath\C{\mathbb{C}}
\defmath\R{\mathbb{R}}
\defmath\T{\mathbb{T}}
\defmath\Q{\mathbb{Q}}
\defmath\N{\mathbb{N}}
\defmath\Z{\mathbb{Z}}

\defmath\LL{\mathcal{L}}
\defmath\Linear{{L}}
\defmath\H{\mathcal{H}}
\defmath\HG{\H/\Gamma}
\defmath\h{\mathfrak{h}}

\defmath\inv{^{-1}}

\defmath\cplus{c^{+}}
\defmath\cminus{c^{-}}

\defmath\Eu{E^u}
\defmath\Es{E^s}
\defmath\Ec{E^c}
\defmath\Ecs{E^{cs}}
\defmath\Ecu{E^{cu}}

\defmath\Euf{E^u_f}
\defmath\Esf{E^s_f}
\defmath\Ecf{E^c_f}
\defmath\Ecuf{E^{cu}_f}
\defmath\Ecsf{E^{cs}_f}
\defmath\Eug{E^u_g}
\defmath\Esg{E^s_g}
\defmath\Ecg{E^c_g}
\defmath\Ecug{E^{cu}_g}
\defmath\Ecsg{E^{cs}_g}

\defmath\Wuf{W^u_f}
\defmath\Wsf{W^s_f}
\defmath\Wcf{W^c_f}
\defmath\Wcuf{W^{cu}_f}
\defmath\Wcsf{W^{cs}_f}
\defmath\Wusf{W^{us}_f}
\defmath\Wug{W^u_g}
\defmath\Wsg{W^s_g}
\defmath\Wcg{W^c_g}
\defmath\Wcug{W^{cu}_g}
\defmath\Wcsg{W^{cs}_g}
\defmath\Wusg{W^{us}_g}

\defmath\piu{\pi^u}
\defmath\pis{\pi^s}

\defmath\gammahat{\hat\gamma}
\defmath\Cph{C_{\textrm{ph}} }

\defmath\fstar{{f_0}_*}
\defmath\gstar{{g_0}_*}
\defmath\fstark{{f^k_0}_*}
\defmath\gstark{{g^k_0}_*}

\defmath\id{\operatorname{\textit{id}} }
\defmath\image{\operatorname{image}}
\defmath\supp{\operatorname{supp}}
\defmath\vol{\operatorname{volume}}

\defmath\piN{\pi_1(N)}

\defmath\hth{h_\theta}

\defmath\W{W}
\defmath\Ws{\W^s}
\defmath\Wu{\W^u}
\defmath\Wc{\W^c}
\defmath\Wcs{\W^{cs}}
\defmath\Wcu{\W^{cu}}
\defmath\Wus{\W^{us}}

\defmath\yux{y\in\Wuf(x)}
\defmath\ysx{y\in\Wsf(x)}
\defmath\ycx{y\in\Wcf(x)}

\defmath\sigmaL{\sigma_\Lambda}
\defmath\sigmaG{\sigma_\Gamma}
\defmath\alphaG{\alpha_\Gamma}

\defmath\RZ{R_{\Z}}
\defmath\spacearrow{\quad\Rightarrow\quad}
\defmath\spacedblarrow{\quad\Leftrightarrow\quad}

\defmath\diam{\operatorname{diam}}
\defmath\dist{\operatorname{dist}}
\defmath\length{\operatorname{length}}
\defmath\height{\operatorname{height}}
\defmath\Area{\operatorname{Area}}
\defmath\disp{\operatorname{disp}}
\defmath\Dom{\operatorname{Dom}}

\defmath\sumc{\sideset{}{^c}\sum}

\newcommand{\deldel}[1]{\frac{\partial}{\partial#1}}

\defmath\eqdef{\buildrel \text{def}\over =}

\newcommand{\BB}[3]{B[{#1,#2,#3}]}

\newcommand{\bigabs}[1]{\bigl|#1\bigr|}

\newcommand{\commdiag}[8]{
    \begin{CD}
           {#1}  @> {#2} >>    {#3}  \\
        @VV{#4}V            @VV{#5}V \\
           {#6}  @> {#7} >>    {#8}
    \end{CD}
}

\newcommand{\blockmat}[4] {
    \left(
    \begin{array}{c|c}
        {#1} & {#2} \\ \hline
        {#3} & {#4}
    \end{array}
    \right)
}

\newcommand{\GPS} {Global Product Structure}
\newcommand{\CL} {Central Shadowing Lemma}

\newtheorem{thm}{Theorem}[section]
\newtheorem{cor}[thm]{Corollary}
\newtheorem{lemma}[thm]{Lemma}

\newtheorem{prop}[thm]{Proposition}

\theoremstyle{remark}
\newtheorem*{remark}{\bf Remark}

\begin{document}

\begin{abstract}
Every partially hyperbolic diffeomorphism on a 3-dimensional nilmanifold is
leaf conjugate to a nilmanifold automorphism.  Moreover, if the nilmanifold is
not the 3-torus, the center foliation is an invariant circle bundle.
\end{abstract}

\maketitle

% Document {{{1
% Introduction {{{1
\section{Introduction} 

A defining achievement of the study of dynamical systems was the
classification of all Anosov systems on nilmanifolds due to J.~Franks and
A.~Manning.
A diffeomorphism $f:M\to M$ is {\em Anosov} if
the tangent bundle of the manifold $M$ splits into two invariant
subbundles
\[ TM=\Eu\oplus\Es,\]
the unstable subbundle $\Eu$ which
is strongly expanded by $Tf$ and the stable subbundle $\Es$ which is strongly
contracted.  A {\em nilmanifold} is a manifold constructed as the quotient
space of a nilpotent Lie group.  A {\em nilmanifold isomorphism} is a
homeomorphism between nilmanifolds that is the quotient of a Lie group
isomorphism on the covering spaces.  The classification result is that every
Anosov diffeomorphism on a compact nilmanifold is topologically conjugate to a
nilmanifold isomorphism \cite{Franks2}\cite{Franks1}\cite{Manning}.
In other words, an
infinitesimal condition on the derivative dictates the global behaviour of the
diffeomorphism, and the algebraic structure of the manifold gives an algebraic
structure to the maps acting on it.

It is conjectured that all Anosov diffeomorphisms occur on nilmanifolds or on
spaces finitely covered by nilmanifolds, and thus that all of these systems can
be understood in terms of this algebraic classification.  To broaden our
understanding of dynamics in general, we must look beyond the uniform
hyperbolicity of Anosov systems.  One such generalization is the notion of
partial hyperbolicity.  A diffeomorphism is {\em partially hyperbolic} if the
tangent bundle splits into three subbundles
\[ TM = \Eu\oplus\Ec\oplus\Es. \]
Here, the center subbundle $\Ec$ may contract or expand slightly, but it
is dominated by the strong expansion and contraction of the unstable and
stable subbundles.
A precise definition will be given in Section \ref{defs}.

In contrast to the stable and unstable directions, there is almost no
restriction on the dynamics in the center direction, and thus, in general, the
center is badly behaved.  To attempt a reasonable classification in the style
of Franks and Manning, we must ignore the ``unhyperbolic'' behaviour along the
center direction.  Suppose $f:M\to M$ and $g:M\to M$ are partially hyperbolic
and there is a foliation tangent to the center subbundle of each of the two
diffeomorphisms.  A {\em (center) leaf conjugacy} between $f$ and $g$ is a
homeomorphism $h:M\to M$, such that for every center leaf $\LL$ of $f$,
$h(\LL)$ is a center leaf of $g$ and
\[hf(\LL)=gh(\LL).\]
A leaf conjugacy does not preserve the arbitrary dynamics which happens along
center leaves, but it does preserve the hyperbolic dynamics which acts on the
center leaves.

The most easily understood partially hyperbolic systems occur on
manifolds of dimension three, where each of the subbundles $\Eu$, $\Ec$ and
$\Es$ is one-dimensional.  In this paper, we prove the following
classification result.

\begin{thm} \label{mainthm}
    Any partially hyperbolic diffeomorphism on a compact, three-\\dimensional
    nilmanifold is leaf conjugate to a nilmanifold automorphism.
\end{thm}

There are two nilpotent, three-dimensional Lie groups, and so there are two
families of compact, three-dimensional nilmanifolds.  The first group consists
only of the 3-torus and the nilmanifold automorphisms on the 3-torus are
exactly the toral automorphisms given by $3\times 3$ invertible matrices with
integer entries.  Theorem \ref{mainthm} for the specific case of the 3-torus
was proven in \cite{ham-thesis}.
Hence, this paper deals exclusively with nilmanifolds in the second family,
given by quotients of the Heisenberg group, $\H$, the Lie group of matrices of
the form
    \[\begin{pmatrix}
        1 & x & z \\ 0 & 1 & y \\ 0 & 0 & 1
    \end{pmatrix}\quad(x,y,z\in\R)\]
under the group operation of matrix multiplication.
For purposes of brevity, we will often write the above triangular matrix as
$(x,y,z)\in\H$.  In this notation, the group operation is
\[
    (a,b,c)\cdot(x,y,z) = (a+x, b+y, z+c+ay).
\]

The prototypical example of a nilmanifold is the quotient of $\H$ by a lattice
consisting of matrices with integer entries.  A fundamental domain of this
quotient is the unit ``cube''
\[
    \{(x,y,z)\in\H : 0 \le x,y,z \le 1 \}.
\]
\begin{figure}[t]
\begin{center}
\includegraphics{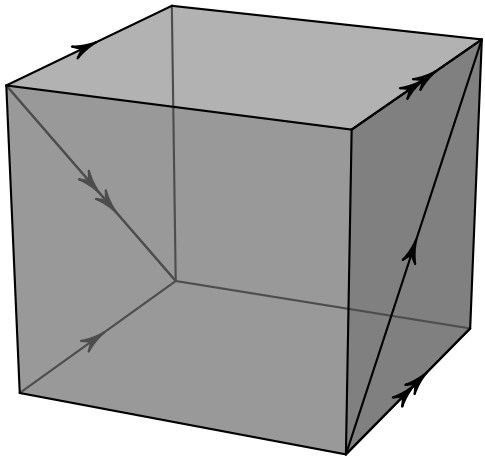}
\end{center}
\caption{A nilmanifold can be constructed from a cube by identifying left and
right faces in a slanted manner.  The other faces are identified by simple
translation.}
\end{figure}
To obtain the nilmanifold from the cube, identify the top face with the
bottom face and the front face with the back face, as one would in
constructing the 3-torus.  The group multiplication
\[
    (1,0,0)\cdot(x,y,z) = (x+1,y,z+y)
\]
tells us to join left and right faces of the cube by the identification
\[
    (0,y,z) \sim (1,y,z+y) \quad \text{(mod $1$)}.
\]
This slanted identification in the $x$-direction frustrates analysis on the
nilmanifold.  To overcome this, we almost exclusively work on the universal
cover, $\H$, and with sets which are bounded in the $x$-direction.

\medskip

To construct an example of a nilmanifold automorphism which is partially
hyperbolic, first consider the Lie group automorphism
\[
    \Phi:\H\to\H,\quad
    (x,y,z)\mapsto(2x+y, x+y, z + x^2 + \tfrac{1}{2} y^2 + xy).
\]
This maps the lattice
\[
    \Gamma = \{ (x,y,z)\in\H : x,y\in\Z,\ z\in\tfrac{1}{2}\Z \}
\]
onto itself, and so $\Phi$ quotients down to a nilmanifold automorphism
$\Phi_0:\HG\to\HG$.  The derivative at $(0,0,0)$, and hence at any point, is
given by the matrix
    \[\begin{pmatrix}
        2 & 1 & 0 \\ 1 & 1 & 0 \\ 0 & 0 & 1
    \end{pmatrix}.\]
This has eigenvalues $\lambda>1>\lambda^{-1}$.  Associating the three
corresponding eigenspaces with $\Eu$, $\Ec$, and $\Es$ respectively, one sees
that $\Phi_0$ is partially hyperbolic.  The $\Ec$ direction is given by the
vector field $\deldel{z}$.  The center foliation is a vertical foliation in
the cube where, due to the simple identification of the top face of the cube
with the bottom, all leaves are circles.

For every linear automorphism, $L:T_0\H \to T_0\H$, of the tangent space of
the identity, there is a group automorphism $\H \to \H$ of the form
\[ \Phi(x,y,z) = (A(x,y),\ cz + p(x,y)) \]
where $A$ is a $2\times 2$ matrix, $c\in\R$, and $p$ is a
quadratic polynomial, and such that the derivative $d\Phi$ at the identity is
equal to $L$.  Given $L$, the corresponding $\Phi$ can always be found by
expanding the equations $d\Phi|_0=L$ and $\Phi(u) \cdot \Phi(v) = \Phi(u \cdot
v)$, then solving for the coefficients of $p$.
Since every Lie group automorphism is uniquely determined by its derivative at
the identity, every automorphism of the Heisenberg group must be of this form.
Such an automorphism $\Phi$ is partially hyperbolic if and only if the matrix
$A$ is hyperbolic.

Any compact nilmanifold constructed from the Heisenberg group is
diffeomorphic to $\HG$ where the lattice
$\Gamma$ is of the form
\[
    \Gamma = \{ (x,y,z)\in\H : x,y\in\Z,\ z\in\tfrac{1}{k}\Z \}
\]
for some $k$.  In this light, Theorem \ref{mainthm} may be restated as
follows.
\begin{cor} \label{maincor}
    Any partially hyperbolic diffeomorphism $f$ on a compact, three-dimensional
    nilmanifold (other than the 3-torus) is topologically conjugate to a
    homeomorphism of the form
    \[
        N\to N,\quad (x,y,z)\mapsto(A(x,y), g(x,y,z))
    \]
    where $N$ is the unit cube $[0,1]^3$ with some identification of faces,
    $A$ is a hyperbolic toral automorphism (on the 2-torus), and $g$ is a
    continuous function $N\to\R/\Z$.

    Moreover, all of the center leaves of $f$ are circles and the
    conjugacy maps the center foliation of $f$ to the vertical foliation on
    $[0,1]^3$.
\end{cor}

Implicit in the statement of Theorem \ref{mainthm} is that any partially
hyperbolic diffeomorphism on a compact, three-dimensional nilmanifold has a
foliation tangent to the center subbundle, as is necessary in the definition of
a leaf conjugacy.  The author proved this condition while preparing this
paper and later discovered that it was independently proved by
K.~Parwani using different techniques \cite{Parwani}.  In the interests of
keeping the paper as self contained as possible, we give a proof of this
result in Section \ref{bbox}.  It arises naturally in the proof of the main
theorem, and does not significantly add to the length of the exposition.

\medskip

The paper is structured as follows.
Section \ref{nilmanifolds} details the algebraic structure of the Heisenberg
group and of functions on that space.  Section \ref{defs} gives a precise
definition of partial hyperbolicity and references known results in the study of
three-dimensional partially hyperbolic systems that will be needed later in the
proof.  Section \ref{bbox} studies the foliations of a partially hyperbolic
systems on the Heisenberg group, culminating in proofs of two key properties
referred to as
\GPS\ (Theorem \ref{GPSthm}) and the \CL\ (Lemma \ref{CLemma}).  Finally,
Section \ref{leafconj} uses these two results to build a leaf conjugacy
between an arbitrary partially hyperbolic system on a nilmanifold, and an
algebraic nilmanifold automorphism.

\section{Nilmanifolds} \label{nilmanifolds} %{{{1
Due to the Bianchi classification, there are exactly two simply connected,
nilpotent Lie groups of dimension three:

\begin{itemize}
    \item the abelian Lie group, $\R^3$, where the group operation is addition,
    and
    \item the Heisenberg group, $\H$, consisting of matrices of the form
    \[\begin{pmatrix}
        1 & x & z \\ & 1 & y \\ & & 1
    \end{pmatrix}\quad(x,y,z\in\R)\]
    under multiplication.  (Here, and in all matrices, blank entries are taken
    to be zero.)
\end{itemize}

The only compact nilmanifold found as a quotient of $\R^3$ is the 3-torus, and
the proof of Theorem \ref{mainthm} for this manifold is given in
\cite{ham-thesis}.
Therefore, from this point on, we take a three-dimensional nilmanifold to mean
a quotient of the form $\HG$ where $\H$ is the Heisenberg group and $\Gamma$
is a discrete, cocompact subgroup, a lattice.  To be specific, $\HG$ is
defined by the equivalence relation
$\gamma\cdot p\sim p$
for $\gamma\in\Gamma$ and $p\in\H$.

The Lie algebra $\h$ associated to the Heisenberg group is generated by
elements
\[
  X = \begin{pmatrix}
        0 & 1 & 0 \\ & 0 & 0 \\ & & 0
      \end{pmatrix},\quad
  Y = \begin{pmatrix}
        0 & 0 & 0 \\ & 0 & 1 \\ & & 0
      \end{pmatrix},\quad\text{and}\quad
  Z = \begin{pmatrix}
        0 & 0 & 1 \\ & 0 & 0 \\ & & 0
      \end{pmatrix}
\]
with $[X,Y]=Z$ and $[X,Z]=[Y,Z]=0$.  The algebra $\h$ may be regarded
as a three-dimensional vector space.  Then, using $X$, $Y$, and $Z$ as a
basis, any Lie algebra endomorphism $\phi:\h\to\h$ may be written as a
$3\times 3$ matrix

\begin{equation}\label{algauto}
  \blockmat
    {A}                {\begin{matrix} 0 \\ 0 \end{matrix}}
    {\begin{matrix} \alpha & \beta \end{matrix}}  {\det(A)}.
\end{equation}
The entries of the last column are determined by the first two columns by the
requirement
$\phi([X,Y]) = [\phi(X),\phi(Y)].$
The Lie algebra endomorphism $\phi$ is invertible if and only if
the associated $3\times 3$ matrix is invertible 
if and only if
the $2\times 2$ submatrix $A$ is invertible.
Let $G$ denote the group of all invertible matrices of the form given in
\eqref{algauto}.

Each Lie algebra automorphism $\phi:\h\to\h$ faithfully corresponds to a Lie
group automorphism $\Phi:\H\to\H.$ Further, if there is a lattice $\Gamma$
such that $\Phi(\Gamma)=\Gamma,$ then $\Phi$
defines a quotient map $\HG\to\HG$ on the nilmanifold.  Conversely, any
endomorphism on the discrete group $\Gamma$ extends uniquely to an
endomorphism on $\H$ \cite{auslander}.

Suppose $f_0:\HG\to\HG$ is a continuous function, and $f:\H\to\H$ is a lift of
$f_0$ to the universal cover.  There is a function $f_*:\Gamma\to\Gamma$ such
that
\[
    f(\gamma\cdot p) = f_*(\gamma)\cdot f(p).
\]
for all $\gamma\in\Gamma$ and $p\in\H$.  $f_*$ is a group endomorphism and
once the lift $f$ is chosen, $f_*$ is unique.  There is a unique Lie group
endomorphism $\Phi:\H\to\H$ such that $\Phi|_\Gamma=f_*$ which we call the
{\em algebraic part} of $f$.

\begin{prop} \label{unifbdd}
    Suppose the Heisenberg group $\H$ is equipped with a metric
    $d(\cdot,\cdot)$ invariant under left-multiplication.  If $f_0:\HG\to\HG$
    is continuous, $f:\H\to\H$ is a lift of $f_0$ to the universal
    cover, and $\Phi:\H\to\H$ is its algebraic part, then the distance between
    $f$ and $\Phi$ is bounded: there is $C>0$ such that
    \[ d\bigl(f(p), \Phi(p)\bigr) < C\]
    for all $p\in\H$.
\end{prop}

\begin{proof}
    First note
    \[  d\bigl(f(\gamma\cdot p), \Phi(\gamma\cdot p)\bigr) =
        d\bigl(f(p), \Phi(p)\bigr)  \]
    for $p\in\H$ and $\gamma\in\Gamma$.
    Then
    \[ 
        \sup_{p\in\H} d\bigl( f(p), \Phi(p) \bigr) =
        \sup_{p\in K} d\bigl( f(p), \Phi(p) \bigr)
    \]
    where $K$ is a fundamental domain of the covering map
    $\H\to\HG$.  As $K$ can be taken as compact, the supremum is finite.
\end{proof}

If the map $f_0:\HG\to\HG$ is a homeomorphism, then so is a lift $f:\H\to\H$
and the algebraic part $\Phi:\H\to\H$ is invertible.  It is a Lie group
automorphism and the corresponding Lie algebra automorphism is represented by
a matrix $T\in G$.  Call $T$ the matrix associated to $f$.

\begin{prop} \label{unit-det}
    If $f_0:\HG\to\HG$ is a homeomorphism, then the matrix
    associated to a lift $f:\H\to\H$ is of the form
    \[
        \blockmat {A} { \begin{matrix} 0 \\ 0 \end{matrix} }
                  { \begin{matrix} \alpha & \beta \end{matrix} } {\pm1}.
    \]
    That is, it is of the form given in \eqref{algauto} with the additional
    property that $\det(A) = \pm 1$.
\end{prop}
\begin{proof}
    Let $\Phi:\H\to\H$ be the algebraic part of $f$.  As a Lie group
    automorphism, $\Phi$ acts as an isomorphism on the (group-theoretic)
    center of $\H$.  The center
    \[
        Z(\H) = \{ (0,0,z) : z\in\R \}
    \]
    is isomorphic to the real line under addition.
    Thus, there is a non-zero factor $a\in\R$ such that $\Phi$ acts on
    $Z(\H)$ as multiplication by $a$.
    The set $Z(\Gamma)=Z(\H)\cap\Gamma$ is a discrete, non-trivial subgroup and
    therefore of the form
    \[
        Z(\Gamma) = \{ (0,0,bz) : z\in\Z \}
    \]
    for some non-zero constant $b$.  Here, $\Phi|_{Z(\Gamma)}$ acts as
    an isomorphism.  This implies that $a=\pm1$
    and $a$ is exactly the entry $\det(A)$ in equation \eqref{algauto}.
\end{proof}

Suppose $\Phi:\H\to\H$ is a Lie group automorphism and $\Gamma\subset\H$ is a
lattice.  Define $\Gamma'$ as the image $\Phi(\Gamma)$.  Then $\Phi$ quotients
down to a homeomorphism $\Phi_0$ between the compact nilmanifolds $\HG$ and
$\HG'$.  Call such a map $\Phi_0$ a {\em nilmanifold isomorphism}.  Call two
homeomorphisms $f_0:\HG\to\HG$ and $g_0:\HG'\to\HG'$ {\em algebraically
conjugate} if there is a nilmanifold isomorphism $\Phi_0:\HG\to\HG'$ such that
\[  \Phi_0 f_0 = g_0 \Phi_0.  \]

\begin{lemma}
    Suppose $f_0:\HG\to\HG$ is a homeomorphism with a lift associated to a
    matrix $S\in G$.
    If $S$ is conjugate to $T\in G$ (that is, there is $P\in G$ such that
    $T=PSP^{-1}$), then $f_0$ is algebraically conjugate to a homeomorphism
    $g_0:\HG'\to\HG'$ with a lift associated to the matrix $T$.
\end{lemma}

\begin{proof}
    Say $P\in G$ is such that $PSP^{-1}=T$.  Then, $P$ induces a Lie group
    automorphism $\Psi:\H\to\H$.  Set $\Gamma'=\Psi(\Gamma)$.  It follows
    that $\Psi$ descends to a nilmanifold isomorphism $\Psi_0:\HG\to\HG'$.
    Define $g_0$ as the composition $\Psi_0 f_0\Psi_0^{-1}$.
    
    Let $f:\H\to\H$ be the lift of $f_0$ associated to $S$, and let
    $\Phi:\H\to\H$ be its algebraic part.  Define $g:\H\to\H$ by
    $g = \Psi f\Psi^{-1}$.  It is a lift of $g_0$, its algebraic part is 
    $\Psi\Phi\Psi^{-1}$, and its associated matrix is $PSP^{-1}=T$.
\end{proof}

Call a matrix $T\in G$ {\em partially hyperbolic} if it has eigenvalues
$\lambda_1,\lambda_2,\lambda_3$ where $|\lambda_1|<1$, $|\lambda_2|>1$ and
$|\lambda_3|=1$.  All eigenvalues are by necessity real, and because the
structure of the matrix is of the form given in \eqref{algauto}, it must be
that the submatrix $A$ is hyperbolic and $\lambda_1\lambda_2=\lambda_3$.

\begin{prop} \label{nicemat}
    Suppose $f_0:\HG\to\HG$ is a homeomorphism with a lift associated to a
    partially hyperbolic matrix $T\in G$.  Then $f_0$ is
    algebraically conjugate to a homeomorphism on a nilmanifold $\HG'$ with a
    lift associated to a diagonal matrix
    \[\begin{pmatrix}
        \lambda_1 & & \\ & \lambda_2 & \\ & & \lambda_3
    \end{pmatrix}\]
    with entries given by the eigenvalues of $T$.
\end{prop}

\begin{proof}
    In light of the previous lemma, we need only show that a partially
    hyperbolic matrix in $G$ is conjugate (in $G$) to a diagonal one.  For
    simplicity, assume that $1$ is an eigenvalue.  The other case, with $-1$
    as an eigenvalue, is handled similarly.

    If
    \[
        T = \blockmat {A}{0}{u}{det(A)} \in G
    \]
    is partially hyperbolic with eigenvalues $\lambda_1, \lambda_2,$ and
    $\lambda_3=1$, it must be that $\det(A)=1$ and that $A$ is hyperbolic.
    Here, $u$ is a $1\times 2$ block.  Let $v$ be another block of the same
    dimensions, and let $I$ denote the $2\times 2$ identity matrix.  Then
    \[
        \blockmat {I}{0}{v}{1} \in G
    \]
    and
    \[
        \blockmat {I}{0}{v}{1}
        \blockmat {A}{0}{u}{1}
        { \blockmat {I}{0}{v}{1} }^{-1} =
        \blockmat {A}{0}{vA+u-v}{1}.
    \]

    As $A$ is hyperbolic, $A-I$ is invertible and there is a unique value of
    $v$ such that $vA+u-v=0$.  Therefore, $T$ is conjugate in $G$ to
    \[ \blockmat {A}{0}{0}{1}. \]
    $A$ is diagonalizable; there is $P\in SL(2,\R)$ such that
    $PAP^{-1}=\begin{pmatrix} \lambda_1 & 0 \\ 0 & \lambda_2
    \end{pmatrix}$.  Then
    \[ \blockmat {P}{0}{0}{1} \in G \]
    and
    \[
        \blockmat {P}{0}{0}{1}
        \blockmat {A}{0}{0}{1}
        { \blockmat {P}{0}{0}{1} } ^ {-1} =
        \blockmat { \begin{matrix}
                        \lambda_1 & 0 \\ 0 & \lambda_2 
                    \end{matrix} }
                  { \begin{matrix} 0 \\ 0 \end{matrix} }
                  { \begin{matrix} 0 & 0 \end{matrix} }
                  {1}
    \]
    as desired.
\end{proof}

The Lie group automorphism associated to the matrix
\[  \begin{pmatrix}
        \lambda_1 & {} & {} \\ {} & \lambda_2 & {} \\ {} & {} & 1
    \end{pmatrix} \] 
is particularly tractable.  It is of the form
\[  \begin{pmatrix}
        1 & x & z \\ & 1 & y \\ & & 1
    \end{pmatrix} \mapsto
    \begin{pmatrix}
        1 & \lambda_1 x & z \\ & 1 & \lambda_2 y \\ & & 1
    \end{pmatrix} \]
with only linear terms.

\section{Definitions and Dependencies} \label{defs} % {{{1
A $C^1$ diffeomorphism $f:M\to M$ on a Riemannian manifold $M$ is {\em
partially hyperbolic} if there is a splitting of $TM$ into three continuous,
non-trivial, $Tf$-invariant sub-bundles
$TM=\Eu\oplus\Ec\oplus\Es$ and constants
$0<\lambda<\gammahat<1<\gamma<\mu$ and $\Cph>1$ such that for $x\in M$
and $n\in\Z$ 
\begin{align*}
  \frac{1}{\Cph}\mu^n\|v\|      \le&\|Tf^nv\|
      &\textrm{for}\ v&\in\Eu(x), \\
  \frac{1}{\Cph}\gammahat^n\|v\|\le&\|Tf^nv\|\le\Cph\gamma^n\|v\|
      &\textrm{for}\ v&\in\Ec(x), \\
  \vphantom{\frac{1}{\Cph}} % for uniform spacing
                                 &\|Tf^nv\|\le\Cph\lambda^n\|v\|
      &\textrm{for}\ v&\in\Es(x).
\end{align*}

\begin{remark}
    The above definition is sometimes called {\em absolute} partial
    hyperbolicity in comparison to ``relative'' or ``pointwise'' partial
    hyperbolicity where the values $\lambda<\gammahat<1<\gamma<\mu$ are
    functions $M\to\R$.
\end{remark}

On a compact manifold, any partially hyperbolic system under one metric will
still be partially hyperbolic under a different choice of metric with at most
a change in the constant $\Cph$.  For nilmanifolds in particular, we work with
one specific metric which we now define.

Recall that an element
\[  \begin{pmatrix}
        1 & x & z \\ & 1 & y \\ & & 1
    \end{pmatrix} \in \H \]
is denoted by the shorthand $(x,y,z)$.  The Lie algebra elements $X$,
$Y$, and $Z$ may be thought of as vector fields on $\H$ invariant under
left-multiplication by elements of the group.  These vector fields are
\[
    X = \deldel{x}, \quad
    Y = \deldel{x}+x\deldel{z},\quad\text{and}\quad
    Z = \deldel{z}.
\]
(The asymmetry between $X$ and $Y$ is due to the one-sidedness of
left-invariance.)
Fix a Riemannian metric on $\H$ such that at each point, the vectors from $X$,
$Y$, and $Z$ form an orthonormal basis of the tangent space.  Call this the
{\em left-invariant metric} on $\H$ and note that it descends to every
nilmanifold $\HG$.  It is with respect to this metric that we consider partial
hyperbolicity.  For points $(x_1,y_1,z_1)$ and $(x_2,y_2,z_2)$ in $\H$, we may
also consider the more familiar the Euclidean metric
\[\|(x_1,y_1,z_1) - (x_2,y_2,z_2)\| =
    \sqrt{(x_1-x_2)^2 + (y_1-y_2)^2 + (z_1-z_2)^2}. \]
There is no meaningful way to project this metric down to a compact nilmanifold
$\HG$, but it will be helpful when studying systems lifted to the Heisenberg
group.  To avoid confusion, we always use $d(p,q)$ for distance in the
left-invariant metric and $\|p-q\|$ for the Euclidean.  Note that while the
two metrics are different, they yield the exact same volume form on $\H$.

\medskip

In a partially hyperbolic system, the unstable $\Eu$ and stable $\Es$
distributions are always uniquely integrable; that is, there are unique
foliations $\Wu$ and $\Ws$ tangent to $\Eu$ and $\Es$ respectively.  The
center direction $\Ec$ is sometimes integrable and sometimes not.  The system
is called {\em dynamically coherent} if $\Ec$, $\Ecu=\Ec\oplus\Eu$ and
$\Ecs=\Ec\oplus\Es$ are uniquely integrable.

In \cite{BBI2}, M.~Brin, D.~Burago, and S.~Ivanov show the remarkable result
that all (absolutely) partially hyperbolic systems on the 3-torus are
dynamically coherent.  A key step in their analysis is showing the absence of
so-called ``transverse contractible cycles'' for three-dimensional systems.

While this paper does not consider transverse contractible cycles directly,
two consequences of their study will be invaluable.  For the following,
suppose $M$ is a closed Riemannian 3-manifold with universal cover $\tilde M$
and $f_0:M\to M$ is a partially hyperbolic diffeomorphism with lift
$f:\tilde M\to\tilde M$.  This lift is also partially hyperbolic, so we may
consider its stable and unstable foliations.

\begin{prop}[Lemma 3.1 of \cite{BBI2}]\label{int-unique}
    If $f_0$ (and therefore $f$) is dynamically coherent, then on the
    universal cover, a leaf of $\Wcs$ and a leaf of $\Wu$ intersect at most
    once.
\end{prop}

In fact, the result in \cite{BBI2} does not {\em a priori} assume
the system is dynamically coherent, and is stated in a slightly different
form.
The uniqueness of intersection immediately gives a topological property of
the leaves.

\begin{cor} \label{propembed}
    If $f$ is dynamically coherent, each center-stable leaf is properly
    embedded on the universal cover.
\end{cor}

\begin{proof}
    If a center-stable leaf $\LL$ accumulated on a point $p\in\H$, then, as the
    foliations are transverse, $\Wu(p)$ would intersect $\LL$ an infinite
    number of times.
\end{proof}

By the same logic, the unstable leaves are properly embedded as well, but we
can say something further.
For a finite-length subcurve $J$ of an unstable leaf, let $U_1(J)$ denote the
set of all points $p\in\tilde M$ such that $\dist(p, J) < 1$.

\begin{prop}[Lemma 3.3 of \cite{BBI2}]\label{uvol}
    There is a constant $C$ such that, for every unstable curve $J$ on the
    universal cover, one has
    \[\vol\ U_1(J)\,\ge\,C\cdot\length(J).\]
\end{prop}

Using $\mu$ as given in the definition of partial hyperbolicity, this result
can be generalized to apply to iterates of the curve.

\begin{cor}\label{fuvol}
    There is a constant $C$ such that, for every unstable curve $J$ on the
    universal cover and every $n\ge 0$, one has
    \[\vol\ U_1(f^n(J))\,\ge\,C\mu^n\cdot\length(J).\]
\end{cor}

In the case of the Heisenberg group, the neighbourhood $U_1$ can be defined
where distance is measured with the left-invariant metric, but not with the
Euclidean metric, since only the former quotients down to the nilmanifold.

Additionally, we need a result proven by 
F.~Rodriguez Hertz, M.~A.~Rodriguez Hertz, and R.~Ures
when establishing ergodicity for measure-preserving partially hyperbolic
systems on three-dimensional nilmanifolds.

\begin{prop}[Proposition 7.2 of \cite{RHRHU-nil}]\label{phmat}
    If $f_0:\HG\to\HG$ is partially hyperbolic with a lift $f:\H\to\H$, then
    the matrix associated to $f$ is partially hyperbolic.
\end{prop}

As stated in their paper, the proposition only holds under certain additional
assumptions which are now known to be unnecessary.  To remove these
assumptions, we adapt the proof slightly.

\begin{proof}
    By Proposition \ref{unit-det}, the matrix associated to $f$ has at least
    one eigenvalue of modulus one.  If the matrix is not partially hyperbolic,
    all three eigenvalues must have modulus one.

    At the core of the proof given in \cite{RHRHU-nil} is the following fact:
    \begin{quote}
        If all eigenvalues are of modulus at most one, and if $J\subset\H$ is
        an unstable curve of $f$, there is a polynomial $p_1$ such that for
        all positive integers $n$
        \[ \diam\ f^n(J)\,\le\,p_1(n).\]
    \end{quote}
    Therefore, $\diam\ U_1(f^n(J)) \le p_2(n)$ for another
    polynomial, $p_2$.  As three-dimension\-al nilmanifolds have polynomial
    growth of volume (see, for instance, Lemma 6.4 again in \cite{RHRHU-nil}),
    there is a polynomial $p_3$ such that
    \[ \vol\ U_1(f^n(J))\,\le\,p_3(n).\]
    This contradicts the exponential growth required by Corollary \ref{fuvol}.
\end{proof}

Finally, we will need a result by J.~Franks to establish a semi-conjugacy
between a partially hyperbolic system on a nilmanifold and a two-dimensional
hyperbolic toral automorphism.  This will be discussed in more detail in
Section \ref{leafconj} when the time comes for its use.

\medskip

The proof of Theorem \ref{mainthm} proceeds by gradually establishing
increasingly stronger properties about the foliations on the universal cover.
One such property is \GPS.
An Anosov system, lifted to the universal cover, has \GPS\ if every stable
leaf intersects every unstable leaf exactly once.  It is not entirely clear
how best to extend this notion to partially hyperbolic systems.  We use the
following definition.
\begin{quote}
A partially hyperbolic system $f:\tilde M\to\tilde M$ has {\em \GPS} if it is
dynamically coherent and the following four properties hold:
\begin{enumerate}
    \item For $p,q\in\tilde M$,\ \ $\Wu(p)$ and $\Wcs(q)$ intersect exactly
          once.
    \item For $p,q\in\tilde M$,\ \ $\Ws(p)$ and $\Wcu(q)$ intersect exactly
          once.
    \item For $p$ and $q$ on the same center-unstable leaf,
          $\Wu(p)$ and $\Wc(q)$ intersect exactly once.
    \item For $p$ and $q$ on the same center-stable leaf,
          $\Ws(p)$ and $\Wc(q)$ intersect exactly once.
\end{enumerate}
\end{quote}
This \GPS\ allows us to understand the foliations at large
scales in a similar way to transverse foliations in small neighbourhoods.

\begin{thm} \label{GPSthm}
    If $f_0:\HG\to\HG$ is partially hyperbolic, then a lift $f:\H\to\H$ has
    \GPS.
\end{thm}

Anosov systems are expansive.  For any such system $f:M\to M$, there is an
$\epsilon>0$ such that
\[
    d(f^n(p),f^n(q)) < \epsilon
\]
for all $n\in\Z$ if and only if the points $p$ and $q$ coincide.

To generalize this notion to partially hyperbolic systems, we must somehow
account for the possibly unexpansive behaviour along the center direction.
One good candidate is the notion of ``plaque expansiveness'' first developed
in \cite{HPS}.  In this paper, we instead exploit the algebraic structure of
the Heisenberg group to give a specific generalization of expansiveness.

\begin{lemma}[\CL] \label{CLemma}
    Let $\H$ be the Heisenberg group with the projection
    \[  P:\H\to\R^2,  \quad  (x,y,z)\mapsto(x,y). \]
    Suppose $f_0:\HG\to\HG$ is partially hyperbolic with lift $f:\H\to\H$.
    Then $p,q\in\H$ lie on the same center leaf if and only if
    $ \|Pf^n(p) - Pf^n(q)\| $
    is bounded for all $n\in\Z$.
\end{lemma}

This says that if the projected orbits of two points on the universal cover
shadow each other, then the two points must lie on the same center leaf.

\medskip

In Section \ref{bbox}, we establish \GPS\ and prove the \CL\ for systems on
three-dimensional nilmanifolds.  Then in Section \ref{leafconj}, using
these two properties, we prove Theorem \ref{mainthm}, the main
classification result.  These two remaining sections are structured to be as
independent as possible, and may be read in either order.

\section{Bounding Boxes} \label{bbox} %{{{1

Throughout this section, assume $f_0:\HG\to\HG$ is partially hyperbolic and
$f:\H\to\H$ is a lift of $f_0$ to the universal cover.  
By Proposition \ref{phmat}, the matrix associated to $f$, has eigenvalues
$\pm\lambda^{-1}$, $\pm\lambda$, and $\pm1$ for some $\lambda>1$.
The two properties to be proved in this section, namely \GPS\ and the \CL,
hold true for $f$ if and only if they hold true for $f^2$.  Therefore, without
loss of generality, assume that the above eigenvalues are positive.

We further wish to compare the system to a Lie group automorphism which is
linear.
To realize this, we will replace $f_0$ by an algebraically conjugate system (as
defined in Section \ref{nilmanifolds}).
First, we must show that such a conjugation does not affect the properties we
wish to prove.

\begin{prop}
    Suppose diffeomorphisms $f_0:\HG\to\HG$ and $g_0:\HG'\to\HG'$ are
    algebraically conjugate.  Each of the following properties holds for $f_0$
    if and only if it holds for $g_0$:
    \begin{enumerate}
        \item partial hyperbolicity,
        \item dynamical coherence,
        \item \GPS,
        \item the \CL.
    \end{enumerate}
\end{prop}

\begin{proof}
    For the first three items, the proof is immediate, as the subbundles of
    the partially hyperbolic splitting and any foliations tangent to these
    subbundles are unaffected when pushed forward by a smooth conjugacy.

    The last item requires some explanation.  Suppose $f_0$ is algebraically
    conjugate to $g_0$.  Then on the universal cover, there are  lifts $f$ and
    $g$ satisfying $g=\Phi f \Phi^{-1}$ for some Lie group automorphism
    $\Phi:\H\to\H$.  As previously noted, $\Phi$ is of the form
    \[  \Phi(x,y,z) = \bigl(A(x,y), cz + p(x,y)\bigr) \]
    for a linear map $A$ on $\R^2$, a constant $c\in\R$, and some polynomial
    $p:\R^2\to\R$.  In particular, the projection $P(x,y,z)=(x,y)$ used in the
    statement of the \CL\ satisfies the relation $P\Phi=AP$.  Then,
    \begin{align*}
        \|Pg^n(p) &- Pg^n(q)\|
            \quad\text{is bounded for all $n\in\Z$.} \spacedblarrow \\
        \|P\Phi f^n\Phi^{-1}(p) &- P\Phi f^n\Phi^{-1}(q)\|
            \quad\text{is bounded for all $n\in\Z$.} \spacedblarrow \\
        \|APf^n\Phi^{-1}(p) &- APf^n\Phi^{-1}(q)\|
            \quad\text{is bounded for all $n\in\Z$.} \spacedblarrow \\
        \|Pf^n\Phi^{-1}(p) &- Pf^n\Phi^{-1}(q)\|
            \quad\text{is bounded for all $n\in\Z$.}
    \end{align*}
    As $\Phi^{-1}$ maps center leaves of $g$ to center leaves of $f$,
    this shows that the \CL\ holds for $f$ if and only if it holds for $g$.
\end{proof}

We are now free to algebraically conjugate the system.  By Proposition
\ref{nicemat}, assume that
    \begin{equation} \label{diaglambda}
        \begin{pmatrix}
            \lambda\smallinv & & \\ & \lambda & \\ & & 1
        \end{pmatrix}
    \end{equation}
is the matrix associated to a lift $f:\H\to\H$ of $f_0$.  This matrix
corresponds to the Lie group automorphism
\begin{equation} \label{deflinear}
    \Linear:\H\to\H, \quad (x,y,z)\mapsto(\lambda^{-1}x,\lambda y,z).
\end{equation}
By Proposition \ref{unifbdd}, the distance $d(f(p),\Linear(p))$ is uniformly
bounded for all $p\in\H$.

This bound holds when measured with the left-invariant metric, a natural metric
to consider on Heisenberg space.
As humans living in more-or-less Euclidean space, however, this metric is a pain
to understand.  Doubters of the difficulties involved are invited to
calculate a formula for $d\bigl( (0,0,0), (0,0,z) \bigr)$ from the
definition of $d$.  (Hint: It's not $|z|$.)  As such, we would like to switch
our analysis from the left-invariant metric to the Euclidean metric as quickly
as possible.  The next few steps show that this is indeed possible.

Define projections $\pis(x,y,z)=x$ and $\piu(x,y,z)=y$ and note the following.

\begin{lemma} \label{distcomp}
    For $p,q\in\H$,
    \begin{align*}
        |\pis(p)-\pis(q)| &\le d(p,q)\quad\text{and} \\
        |\piu(p)-\piu(q)| &\le d(p,q).
    \end{align*}
\end{lemma}
\begin{proof}
    We prove the first inequality.  The proof for the second equality is
    essentially the same.
    Let $\alpha:[0,1]\to\H$ be a $C^1$ path from $p$ to $q$.  At each point,
    $\alpha'(t)$ splits into components
    \[ \alpha_X'(t), \quad \alpha_Y'(t), \quad\text{and}\quad \alpha_Z'(t) \]
    such that
    \[ \alpha'(t) =
           \alpha_X'(t) X_{\alpha(t)} + 
           \alpha_Y'(t) Y_{\alpha(t)} + 
           \alpha_Z'(t) Z_{\alpha(t)}. \]
    By the Fundamental Theorem of Calculus,
    \[ q-p = \alpha(1) - \alpha(0) = \int_0^1 \alpha'(t) dt \]
    where the subtraction is coordinate-wise.
    When projecting to the first component, this yields
    \[ \pis(q) - \pis(p) = \int_0^1 \alpha_X'(t) dt.\]
    The length of $\alpha$ with respect to the left-invariant metric is
    \[ \length(\alpha) = \int_0^1 \sqrt{
           (\alpha_X'(t))^2 +
           (\alpha_Y'(t))^2 +
           (\alpha_Z'(t))^2 } dt. \]
    Therefore,
    \[ \length(\alpha) \ge \int_0^1 |\alpha_X'(t)| dt \ge
           \left| \int_0^1 \alpha_X'(t) dt \right| =
           \bigl|\pis(q)-\pis(p)\bigr|. \]
    As $d(p,q)$ is the infimum of such lengths of paths, the lemma is proved.
\end{proof}

\begin{cor} \label{xbound}
    There is $x_0>0$ such that for $p\in\H$
    \[ |\pis(p)| \le x_0 \spacearrow |\pis f(p)| \le x_0. \]
\end{cor}
\begin{proof}
    Recall that for $\Linear:\H\to\H$ as defined in \eqref{deflinear}, the
    distance between $f(p)$ and $\Linear(p)$ is uniformly bounded, say by
    $C>0$.
    The definition also implies that $\pis \Linear(p) = \lambda^{-1}\pis(p)$
    and as $\lambda>1$, there is $x_0>0$ such that
    $\lambda^{-1} x_0 + C < x_0$.  Then,
    \begin{align*}
        \bigl|\pis\bigl(p\bigr)\bigr| &\le x_0 \spacearrow \\
        \bigl|\pis f(p)\bigr|
                     &\le \bigl|\pis \Linear(p)\bigr| +
                          \bigl|\pis f(p) - \pis \Linear(p)\bigr| \\
                     &\le \bigl|\pis \Linear(p)\bigr| + d\bigl(f(p), \Linear(p)\bigr) \\
                     &\le \lambda^{-1} x_0 + C < x_0.
    \end{align*}
\end{proof}

For $0 < x_0,y_0,z_0 \le \infty$, define the set
\[ \BB{x_0}{y_0}{z_0} =
   \{ (x,y,z)\in\H : |x|\le x_0, |y|\le y_0, |z|\le z_0 \}. \]
The last corollary may be restated as
\[ f(\BB{x_0}{\infty}{\infty}) \subset \BB{x_0}{\infty}{\infty}. \]

Note that if we restrict our analysis to $\BB{x_0}{\infty}{\infty}$, the
``Euclidean length'' of the basis vector $Y=\deldel{y} + x \deldel{z}$ is
bounded and therefore distances in the left-invariant metric and the Euclidean
metric are comparable.
    Fixing $x_0>0$, for every $a>0$ there is $b>0$ such that
    \begin{equation} \label{complength}
        d(p,q) < a \spacearrow \|p-q\| < b
    \end{equation}
    for all $p,q\in\BB{x_0}{\infty}{\infty}$.
Restricted to this region, $f$ and the linear map $\Linear$ are a bounded distance
apart, not just when measured with the left-invariant metric, but with the
Euclidean metric as well.  This gives a means to measure how quickly subsets
of Heisenberg space grow under iterates of $f$.

\begin{lemma} \label{boxgrow}
    There is $x_0>0$ and $c>0$ such that
    \[ f(\BB{x_0}{y}{z}) \subset \BB{x_0}{\lambda y+c}{z+c} \]
    for all $y,z>0$.
    Moreover, if $\beta > \lambda$ then there is $y_0>0$ such that
    \[ f^n(\BB{x_0}{y}{z}) \subset \BB{x_0}{\beta^ny}{z+nc} \]
    for all $y > y_0$, $z > 0$ and positive integers $n$.
\end{lemma}
\begin{figure}[p!]
\begin{center}
\includegraphics{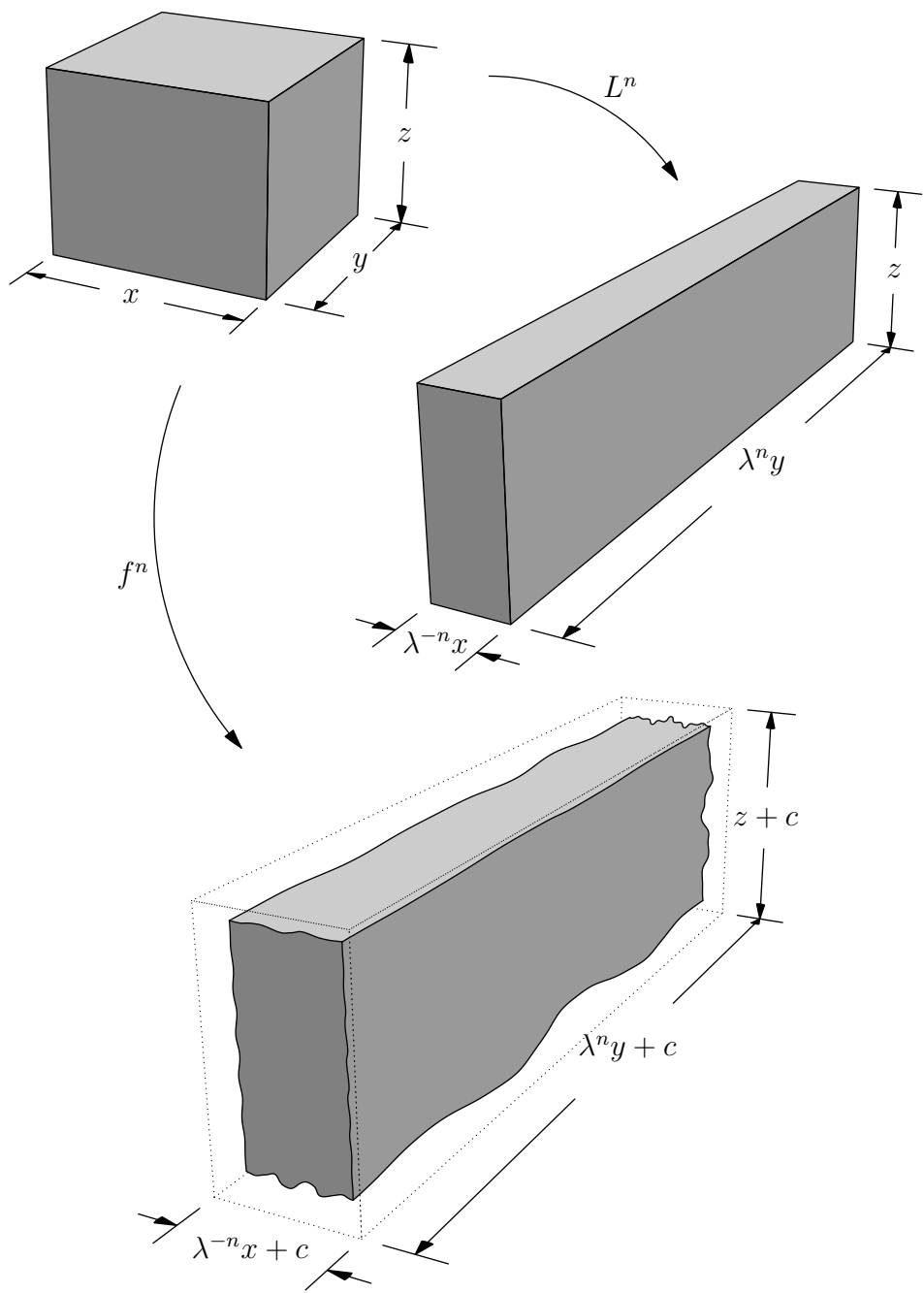}
\end{center}
\caption{The images of a box under iterates of the linear map $\Linear$ and
the diffeomorphism $f$.}
\label{fig:bbox}
\end{figure}
\begin{proof}
    As $d(\Linear(p), f(p))$ is bounded on all of $\H$, it follows from
    \eqref{complength} that there is $c>0$ such that
    \[ \|f(p) - \Linear(p)\| < c \]
    for $p\in\BB{x_0}{\infty}{\infty}$, where $x_0$ is as in Corollary
    \ref{xbound}.  Then, for $y,z>0$,
    \begin{align*}
        \Linear(\BB{x_0}{y}{z}) &= \BB{\lambda^{-1}x_0}{\lambda y}{z} \spacearrow \\
        f(\BB{x_0}{y}{z}) &\subset
                \BB{\lambda^{-1}x_0+c}{\lambda y + c}{z + c}
    \end{align*}
    Since we already know
    $f(\BB{x_0}{\infty}{\infty}) \subset \BB{x_0}{\infty}{\infty}$,
    this establishes the first statement of the lemma.  To prove the second
    statement, just choose $y_0$ large enough that
    \[\beta y_0 > \lambda y_0 + c \]
    and apply induction.
\end{proof}

We now know enough to compare $\lambda$, the expanding eigenvalue
determined by the algebraic part of the diffeomorphism, with $\mu$, the lower
bound on the growth rate of the unstable direction in the definition of
partial hyperbolicity.

\begin{lemma}
    $\mu \le \lambda$
\end{lemma}

\begin{proof}
    Take $\beta$ such that $\lambda<\beta$.
    Let $J$ be a small unstable curve inside a region $\BB{x_0}{y}{z}$.  By
    the previous lemma,
    \begin{align*}
        f^n(J) &\subset \BB{x_0}{\beta^ny}{z+nc} \spacearrow \\
        U_1(f^n(J)) &\subset \BB{x_0+d}{\beta^ny+d}{z+nc+d}
    \end{align*}
    for constants $c,d>0$.  Volume in Heisenberg space is computed exactly as
    in Euclidean space, and therefore
    \[ \vol\ U_1(f^n(J)) < \beta^np(n) \]
    for some polynomial $p$.  By Corollary \ref{fuvol},
    \[ C\mu^n < \vol\ U_1(f^n(J)) \]
    for some constant $C>0$.  The only way this can hold for all $n\ge 0$ is
    if $\mu\le\beta$.  Since $\beta$ is any constant greater than $\lambda$,
    this means $\mu\le\lambda$ as well.
\end{proof}

\begin{prop} \label{yexpand}
    For every $\alpha<\lambda$, there is $M>0$ such that
    \[ \bigabs{\piu(p) - \piu(q)} \ge M \spacearrow
       \bigabs{\piu f^n(p)  - \piu f^n(q) } > \alpha^n \]
    for all $p,q\in\H$ and positive integers $n$.
\end{prop}
\begin{proof}
    Let $C>0$ be such that
    $d \bigl( f(p),\Linear(p) \bigr) <C$ for all $p\in\H$, and
    choose $M>1$ large enough that $\lambda M - 2C > \alpha M$.

    Using the definition of $\Linear$ and Lemma \ref{distcomp},
    \begin{align*}
        \bigabs{\piu(p) - \piu(q)} &\ge M \spacearrow \\
        \bigabs{\piu f(p) - \piu f(q)}
            &\ge \bigabs{\piu \Linear(p)  - \piu \Linear(q)} -
                \bigabs{\piu f(p) -\piu \Linear(p)} \\
            &\phantom{\ge \bigabs{\piu \Linear(p) - \piu \Linear(q)} }
                \ - \bigabs{\piu f(q) -\piu \Linear(q)} \\
            &\ge \lambda\bigabs{\piu(p) - \piu(q)}
                - d(f(p),\Linear(p)) - d(f(q),\Linear(q)) \\
            &\ge \lambda\bigabs{\piu(p) - \piu(q)} - 2C \\
            &> \alpha\bigabs{\piu(p) -\piu(q)}.
    \end{align*}
    By induction,
    \[
        \bigabs{\piu f^n(p)  - \piu f^n(q)}
            > \alpha^n \bigabs{\piu(p) - \piu(q)}
            > \alpha^n.
    \]
\end{proof}

\begin{lemma} \label{curveunbdd}
    For $M>0$ there is $\ell>0$ such that any unstable curve $J$ of length
    greater than $\ell$ contains points $p,q$ with the property
    \[ |\piu(p) - \piu(q)| > M. \]
\end{lemma}
\begin{proof}
    We first prove that for an unstable curve $J$ of length exactly one, some
    iterate $f^n(J)$ satisfies the above property, and that $n$ is independent
    of the choice of $J$.

    Let $J$ be an unstable curve of length exactly one.  By applying a deck
    transformation, we assume without loss of generality that $J$ is within
    some fixed distance of the origin and therefore that $U_1(J)$, the
    neighbourhood of radius one of $J$, is contained in $\BB{x_0}{\infty}{z_0}$
    where $x_0$ is as in Lemma \ref{boxgrow} and $z_0$ is independent of $J$.

    Let $y^-_n = \inf_{p\in J} \piu(f^n(p))$ and
    $y^+_n = \sup_{p\in J} \piu(f^n(p))$.  Then, using Lemma \ref{boxgrow},
    \begin{align*}
            U_1(f^n(J))
            \subset \{(x,y,z)\in\H :\ 
                & |x|\le x_0, \\
                & y^-_n-1 \le y \le y^+_n-1,\\
                & |z|\le z_0 + nc \}
    \end{align*}
    and therefore the volume is bounded by
    \[
        2x_0\ (y^+_n-y^-_n + 2)\ 2(z_0+cn).
    \]
    From Corollary \ref{fuvol},
    \[ \vol\ U_1(f^n(J)) \ge C\mu^n \]
    for some constant $C>0$.  Thus, $y^+_n-y^-_n > M$ for large $n$, and
    $f^n(J)$ has points $p,q$ such that
    \[ |\piu(p) - \piu(q)| > M. \]
    Specifically, take $n$ to be the smallest integer such that
    \[ 2x_0\ (M+2)\ 2(z_0+cn) \le C\mu^n. \]
    This choice of $n$ is independent of the unstable curve $J$, so long as it
    has unit length.

    To complete the proof of the lemma as stated above, take $\ell>0$ large
    enough that the length of $f^{-n}(J)$ is greater than one for any unstable
    curve $J$ of length greater than $\ell$.
\end{proof}

\begin{cor} \label{udiam}
    For $\alpha<\lambda$ and an unstable curve $J$, there is $C>0$ such that
    \[\diam\, f^n(J)\, >\, C\alpha^n\]
    for all $n \ge 0$.
\end{cor}
\begin{proof}
    This is a combination of Proposition \ref{yexpand} and Lemma
    \ref{curveunbdd}.
\end{proof}

We are at a point now where we need to use that $f$ is dynamically coherent.
Fortunately, we are also at a point where we can prove it.

\begin{thm}
    $f$ is dynamically coherent.
\end{thm}

\begin{proof}
    For a Riemannian manifold $M$, suppose $E\subset TM$ is a continuous
    distribution of codimension one.  For $p\in M$ and $\epsilon>0$ let
    $R_\epsilon(p)$ denote the set of all points reachable from $p$ by a path
    tangent to $E$ and of length less than $\epsilon$.  If $E$ is uniquely
    integrable, then the set $R_\epsilon(p)$ is a small plaque of the leaf
    through $p$.
    If $E$ is not uniquely integrable, there is a point $p$ such that for any
    $\epsilon>0$, the set $R_\epsilon(p)\subset M$ has non-empty interior.

    In our case, suppose $\Ecs\subset T\H$ is not uniquely integrable.  Then,
    for any $\epsilon>0$ there is a point $p\in\H$ and a small unstable curve
    $J$ such that every point on $J$ is reachable from $p$ by a $cs$-path of
    length less than $\epsilon$.  Using the definition of partial
    hyperbolicity, for $q\in J$
    \[ d\bigl(f^n(p), f^n(q)\bigr) \le \Cph\,\gamma^n\,\epsilon \]
    and therefore
    \[ \diam\,f^n(J)\,\le\,2\,\Cph\,\gamma^n\,\epsilon. \]
    This contradicts Corollary \ref{udiam} if we use a value $\alpha$ such
    that
    \[ \gamma < \alpha < \mu \le \lambda. \]
    Thus $\Ecs$ is uniquely integrable.

    This shows that the center-stable bundle of any partially hyperbolic
    diffeomorphism on the nilmanifold is uniquely integrable.  Since $\Ecu$,
    the center-unstable subbundle of $f$, is equal to the center-stable
    subbundle of $f^{-1}$ (also partially hyperbolic), it is
    uniquely integrable as well, as is $\Ec$, the intersection of $\Ecu$ and
    $\Ecs$.
\end{proof}

\begin{lemma} \label{csbdd}
    Center-stable leaves are bounded in the $\piu$ direction: there is $R>0$
    such that for all $p\in\H$ and $q\in\Wcs(p)$
    \[ |\piu(p) - \piu(q)| < R.\]
\end{lemma}
\begin{proof}
    Set $\alpha$ as in the previous proof and let $R=M>0$ be as in Proposition
    \ref{yexpand}.  Then an analysis of the growth rate of
    $|\piu f^n(p) -\piu f^n(q)|$ similar in form to the last proof gives the
    result.
\end{proof}

\begin{remark}
    The last two proofs crucially rely on $f$ being absolutely partially
    hyperbolic in order to find a constant $\alpha$ between $\gamma$ and $\mu$.
    These are the only places we use absolute partial hyperbolicity instead of
    the more general pointwise definition.
\end{remark}

By Corollary \ref{propembed}, each center-stable leaf is a properly embedded
surface which divides $\H$ into two components.  These components can be
thought of as half-spaces in the following sense.

\begin{lemma} \label{halfspaces}
    Let $R$ be as in Lemma \ref{csbdd}.  For any $\Wcs(p)$, the complement
    $\H\setminus\Wcs(p)$ has connected components $A_p$ and $B_p$
    satisfying the following set inclusions:
    \[ H^-(p)\ \eqdef\ \{ q\in\H : \piu(q) < \piu(p) - R \} \subset A_p \]
    and
    \[ H^+(p)\ \eqdef\ \{ q\in\H : \piu(q) > \piu(p) + R \} \subset B_p. \]
\end{lemma}

\begin{figure}[t]
\begin{center}
\includegraphics{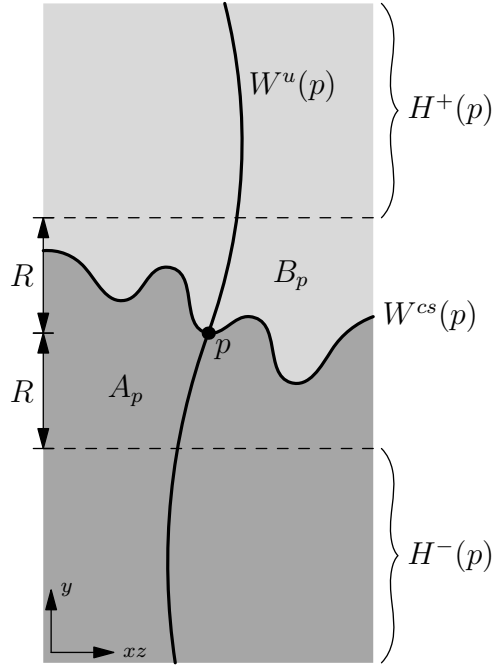}
\end{center}
\caption{A depiction of the half-spaces defined in Lemma \ref{halfspaces}.}
\label{fig:halfspace}
\end{figure}

\begin{proof}
    Let $A$ and $B$ denote the two connected components of
    $\H\setminus\Wcs(p)$.

    By Proposition \ref{int-unique}, $\Wu(p)$ intersects $\Wcs(p)$ only at the
    point $p$ and the intersection is transverse.  Therefore,
    $\Wu(p)\setminus\{p\}$ consists of two infinitely long curves, one of
    which lies entirely in $A$ and the other entirely in $B$.  

    Each curve contains unstable sub-curves of arbitrarily long length, and so
    by Lemma \ref{curveunbdd}, the image of each curve under $\piu$ is
    unbounded.  This shows that both $\piu(A)$ and $\piu(B)$ are unbounded.

    By Lemma \ref{csbdd}, each of $H^-(p)$ and $H^+(p)$ is contained in
    $\H\setminus\Wcs(p)$ and so each is contained wholly in either $A$ or
    $B$.  Suppose both are contained in the same component, say $A$.  Then
    \[
        B \subset \H\setminus\bigl(H^-(p)\cup H^+(p)\bigr)
          = \{q\in\H : |\piu(p) - \piu(q)| \le R\}.
    \]
    This contradicts the fact that $\piu(B)$ is unbounded.  Therefore, either
    $A$ or $B$ contains $H^+(p)$ and the other contains $H^-(p)$.
\end{proof}

\begin{prop} \label{int-exists}
    For $p,q\in\H$, $\Wcs(p)$ and $\Wu(q)$ intersect exactly once.
\end{prop}

\begin{proof}
    Uniqueness of the intersection has already been established in
    Proposition \ref{int-unique}, so we need only prove existence.

    Let $A_p$, $B_p$, $H^+(p)$, and $H^-(p)$ be as in the previous lemma.
    Adopt similar notation for the point $q$.
    Assume $\Wu(q)$ does not intersect $\Wcs(p)$.  Then it lies in one of the
    connected components of the complement, and either
    \[  \Wu(q) \subset A_p \subset \H\setminus H^+(p)
        \quad\text{or}\quad
        \Wu(q) \subset B_p \subset \H\setminus H^-(p) \]
    \begin{figure}[t]
    \begin{center}
    \includegraphics{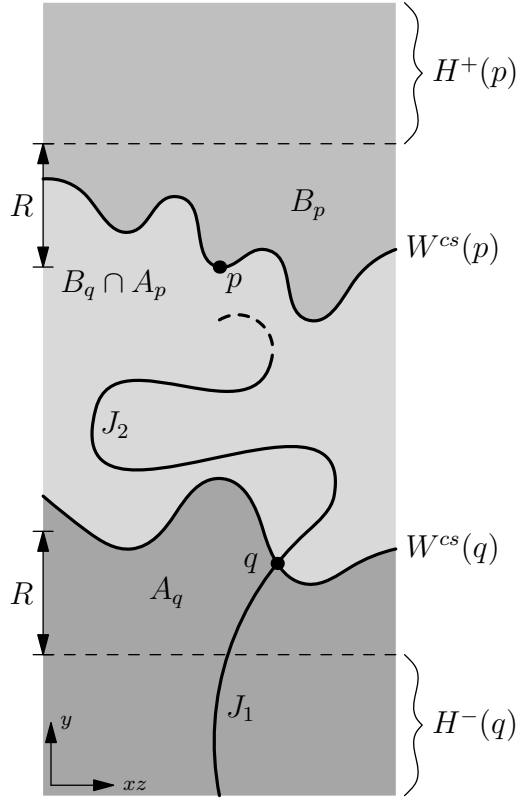}
    \end{center}
    \caption{The unstable leaf though $q$ must intersect the center-stable
    leaf through $p$.  Otherwise, the arc $J_2$ would be bounded in the
    $y$-direction.}
    \label{fig:boundcurve}
    \end{figure}

    Assume, without loss of generality, that the first inclusion holds, as
    depicted in Figure \ref{fig:boundcurve}.

    As in the proof of the last lemma, $\Wu(q)\setminus\{q\}$ consists of
    two unbounded curves lying in the two connected components of
    $\H\setminus\Wcs(q)$.  Label these curves as $J_1$ and $J_2$ where
    \[  J_1 \subset A_q \subset \H\setminus H^+(q)
        \quad\text{and}\quad
        J_2 \subset B_q \subset \H\setminus H^-(q). \]
    Then,
    \begin{align*}
        J_2 &\subset \H\setminus \bigl(H^+(p) \cup H^-(q)\bigr) \\
            &\subset \{s\in\H : \piu(q)-R\le \piu(s)\le\piu(p)+R \},
    \end{align*}
    but, as $J_2$ is an infinitely long unstable curve, $\piu(J_2)$ is
    unbounded, a contradiction.

\end{proof}

This proves the first of the four axioms required to establish \GPS\ (Theorem
\ref{GPSthm}).  To prove the second axiom, consider Proposition
\ref{int-exists} when applied to $f^{-1}$ in place of $f$.
To establish the last two axioms, we must show that on a $cu$-leaf each center
leaf intersects each unstable leaf exactly once.  To do this, repeat the
preceding steps of this section, but restricted to a single center-unstable
leaf.  Specifically, if $\LL\subset\H$ is a center-unstable leaf, show:
\begin{enumerate}
    \item If $p\in\LL$ then $\Wc(p)$ is properly embedded in $\LL$, akin to
    Corollary \ref{propembed}.
    \item If $p\in\LL$ then $\LL\setminus\Wc(p)$ has connected components $A$
    and $B$ such that
    \begin{align*}
        \{q\in\LL : \piu(q) < \piu(p) - R \} &\subset A \quad\text{and} \\
        \{q\in\LL : \piu(q) > \piu(p) + R \} &\subset B
    \end{align*}
    akin to Lemma \ref{halfspaces}.
    \item If $p,q\in\LL$ then $\Wc(p)$ and $\Wu(q)$ intersect exactly once,
    akin to Lemma \ref{int-exists}.
\end{enumerate}

With \GPS\ established, we proceed to prove the \CL\ (Lemma \ref{CLemma}).

\begin{lemma} \label{endsunbdd}
    For $M>0$ there is $\ell>0$ such that for any unstable curve $J$ of
    length greater than $\ell$, the endpoints $p$ and $q$ satisfy
    \[ |\piu(p) - \piu(q)| > M. \]
\end{lemma}
\begin{remark}
    This claim is stronger than Lemma \ref{curveunbdd} because it concerns the
    endpoints specifically, instead of two points somewhere along the curve.
\end{remark}
\begin{proof}
    Let $\alpha:[0,1]\to\H$ be a parametrization of a sufficiently long
    unstable curve.  Then, by Lemma \ref{curveunbdd}, there are $s,t\in[0,1]$ such
    that
    \[ \piu \alpha(t) - \piu \alpha(s) > M + 2R \]
    where $M$ is given in the hypothesis of this lemma, and $R$ is as in
    Lemma \ref{csbdd}.  Without loss of generality, $0<s<t<1$.

    The sub-curve $\alpha((t,1])$ must lie in one connected component of
    $\H\setminus\Wcs(\alpha(t))$ while $\alpha(s)$ lies in the other.
    As such, $\alpha(s)\in H^-(\alpha(t))$ where $H^-$ is as defined in Lemma
    \ref{halfspaces} and therefore $\alpha((t,1]) \subset \H\setminus
    H^-(\alpha(t)).$ Similarly, $\alpha([0,s)) \subset \H\setminus
    H^+(\alpha(s)).$
    Then,
    \[ \piu \alpha(1) - \piu \alpha(0) > M \]
    due to the definitions of $H^+$ and $H^-$ and the triangle inequality.

    \begin{figure}[t]
    \begin{center}
    \includegraphics{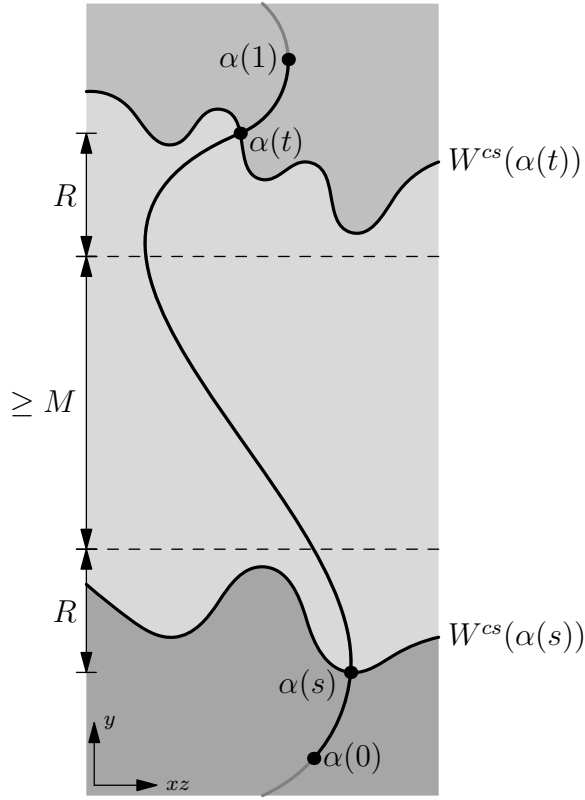}
    \end{center}
    \caption{The curve $\alpha$ considered in the proof of Lemma
    \ref{endsunbdd}.}
    \label{fig:longends}
    \end{figure}

\end{proof}

\begin{remark}
    With a little more work, one could find constants $a,b>0$ such that
    \[  |\piu(p) - \piu(q)| > a\cdot\ell + b \]
    for any unstable curve of length $\ell$ and endpoints $p$ and $q$.  Then,
    by Lemma \ref{distcomp},
    \[  d(p,q) > a\cdot\ell + b \]
    which is precisely what it means for the unstable foliation to be {\em
    quasi-isometric}.  Due to Brin, quasi-isometry of the stable and unstable
    foliations is a sufficient condition for a system to be dynamically
    coherent \cite{Brin}.  Brin, Burago, and Ivanov used this property to show
    that all partially hyperbolic systems on the 3-torus were dynamically
    coherent \cite{BBI2}, and Parwani extended this to all 3-nilmanifolds
    \cite{Parwani}.

    In contrast, the proof of dynamical coherence given in this paper does not
    use quasi-isometry, and may be adaptable to systems where quasi-isometry
    is not known to hold.

    Lemma \ref{endsunbdd} is sufficient for the final goal of this section,
    proving the \CL, and so we do not further pursue the
    idea of quasi-isometry.
\end{remark}

We first prove a baby version of the \CL\ for center-stable leaves.

\begin{lemma} \label{babyshadow}
    For $p,q\in\H$, the following are equivalent:
    \begin{itemize}
        \item $p\in\Wcs(q)$.
        \item $|\piu f^n(p) - \piu f^n(q)|$ is bounded for all $n\in\Z$.
    \end{itemize}
\end{lemma}
\begin{proof}
    By the \GPS, there is a unique point $r\in\H$ such that
    $r\in\Wcs(p)$ and $r\in\Wu(q)$.  By Lemma \ref{csbdd},
    $|\piu f^n(p) - \piu f^n(r)|$ is bounded for all $n\in\Z$.
    Then, (using Lemma \ref{endsunbdd} for the second ``$\Leftrightarrow$'')
    \begin{align*}
        p\in\Wcs(q)& \spacedblarrow \\
        r = q& \spacedblarrow \\
        |\piu f^n(r) - \piu f^n(q)|& \quad\text{is bounded for all $n\in\Z$}
               \spacedblarrow \\
        |\piu f^n(p) - \piu f^n(q)|& \quad\text{is bounded for all $n\in\Z$.}
    \end{align*}
\end{proof}

By considering $f^{-1}$, we prove a similar statement for the
center-unstable leaves, but because our coordinate system is adapted to $f$
and not to $f^{-1}$, the proof requires some care.

\begin{lemma}
    For $p,q\in\H$, the following are equivalent:
    \begin{itemize}
        \item $p\in\Wcu(q)$.
        \item $|\pis f^n(p) - \pis f^n(q)|$ is bounded for all $n\in\Z$.
    \end{itemize}
\end{lemma}
\begin{proof}
    The standing assumption of this section is that $f:\H\to\H$ is a partially
    hyperbolic diffeomorphism with the associated matrix given in
    \eqref{diaglambda}.
    Consider the Lie group automorphism
    \[\Phi:\H\to\H, \quad (x,y,z)\mapsto(-y,x,z-xy)\]
    with associated matrix
    \[  \begin{pmatrix}
            0 & -1 & 0 \\ 1 & 0 & 0 \\ 0 & 0 & 1
        \end{pmatrix}. \]
    Further consider $\Phi f^{-1}\Phi^{-1}$, the conjugation of the
    inverse of $f$ by $\Phi$.  This also has the associated matrix
    \eqref{diaglambda}
    and so Lemma \ref{babyshadow} applies just as well to
    $\Phi f^{-1}\Phi^{-1}$
    as it does to $f$.

    If $\LL$ is a $cu$-leaf of $f$, then $\Phi(\LL)$ is a $cs$-leaf of 
    $\Phi f^{-1}\Phi^{-1}$
    and, therefore, the following are equivalent:
    \begin{itemize}
        \item $p$ and $q$ lie on the same $cu$-leaf of $f$.
        \item $\Phi(p)$ and $\Phi(q)$ lie on the same $cs$-leaf of
        $\Phi f^{-1}\Phi^{-1}$.
        \item $|\piu\Phi f^{-n}(p) - \piu\Phi f^{-n}(q)|$
        is bounded for all $n\in\Z$.
    \end{itemize}
    The definition of $\Phi$ combined with $\pis(x,y,z)=x$ and
    $\piu(x,y,z)=y$
    shows that $\piu\Phi=\pis$,
    implying that the last item above is equivalent to
    \begin{itemize}
        \item $|\pis f^{-n}(p) - \pis f^{-n}(q)|$
        is bounded for all $n\in\Z$
    \end{itemize}
    and the lemma is proved.
\end{proof}

Let $P:\H\to\R^2$ be the projection $(x,y,z)\mapsto(x,y)$.  Due to the Global
Product Structure, two points lie on the same center leaf if and only if they
lie on the same $cu$-leaf and the same $cs$-leaf.  As such, the last two
lemmas combine to show the following are equivalent:
\begin{itemize}
    \item $p\in\Wc(q)$.
    \item $\|Pf^n(p) - Pf^n(q)\|$ is bounded for all $n\in\Z$.
\end{itemize}
This is precisely the \CL (Lemma \ref{CLemma}).

\section{The Leaf Conjugacy} \label{leafconj} %{{{1
In the previous section, we viewed the partially hyperbolic system in such a
way that it could be closely compared to a Lie group automorphism of the form
$(x,y,z)\mapsto(\lambda^{-1}x,\lambda y, z)$.  This was achieved by a change
of coordinates that put the lattice $\Gamma$ defining the nilmanifold $\HG$
into an unknown state.  Since the proofs of the previous section did not
involve the lattice, the change was benign.  In this section, however, to
construct a leaf conjugacy on $\HG$, it is important that $\Gamma$ be as
simple as possible.

Any lattice $\Gamma\subset\H$ defining a nilmanifold $\HG$ is of the form
$\Gamma=\langle a,b,c\rangle$ where $[a,b]=aba^{-1}b^{-1}=c^k$ for some
positive integer $k$.  In fact, $k$ is the index of $[\Gamma,\Gamma]$ as a
subgroup of $[\H,\H]\cap\Gamma$ and depends only on $\Gamma$ and not the
choice of generators.  If $\Gamma'=\langle a',b',c'\rangle$ has the same index
$k$, there is a unique Lie group automorphism $\Phi:\H\to\H$ mapping $a$, $b$,
$c$ to $a'$, $b'$, $c'$ respectively.  This is most easily seen by
pulling the lattices back by the surjective exponential map $\exp:\h\mapsto\H$
and first defining the automorphism on the Lie algebra.

Once defined, the Lie group automorphism descends to a nilmanifold
isomorphism $\HG\to\H/\Gamma'$ which is a smooth diffeomorphism.
Therefore, assume without loss of generality that the partially hyperbolic
system under study is defined on a nilmanifold $\HG$ where the lattice is
given by
\[
    \Gamma = \bigl\langle\ (1,0,0),\ (0,1,0),\ (0,0,\tfrac{1}{k})\ \bigr\rangle
\]
for some positive integer $k$.

Recall that there is a group isomorphism $f_*:\Gamma\to\Gamma$ such that
\[
    f(\gamma\cdot p) = f_*(\gamma)\cdot f(p)
\]
for $\gamma\in\Gamma$ and $p\in\H$. This restricts to an isomorphism on the
(group-theoretic) center $Z(\Gamma)=\{(0,0,\tfrac{i}{k}):i\in\Z\}$,
which implies that 
\[
    f(\gamma\cdot p) = \pm\gamma\cdot f(p)
\]
for all $\gamma\in Z(\Gamma)$.

\begin{prop} \label{cmptleaf}
    For all $p\in\H$, the point $(0,0,\tfrac{1}{k})\cdot p$ lies on the same
    center leaf as $p$.
\end{prop}
\begin{proof}
    Let $\gamma=(0,0,\tfrac{1}{k})$.
    For $p\in\H$, $f(\gamma\cdot p)=\pm\gamma\cdot f(p)$ and by
    induction, $f^n(\gamma\cdot p) = \pm\gamma\cdot f^n(p)$ for all $n\in\Z$.
    The distance $d(\pm\gamma\cdot f^n(p), f^n(p))$ is bounded independently
    of $n$, and therefore, by the \CL, $p$ and $\gamma\cdot p$ lie on the same
    center leaf.
\end{proof}

Multiplication by $(0,0,\tfrac{1}{k})$ is given by
\[
    (0,0,\tfrac{1}{k})\cdot(x,y,z) = (x,y,z+\tfrac{1}{k}).
\]
This shows that the center foliation of $f$ is more-or-less vertical, and when
we quotient down to the nilmanifold $\HG$, all of the center leaves are
circles.

Consider the projection $P:\H\to\R^2,\ (x,y,z)\mapsto(x,y)$.  As
$P(\Gamma)=\Z^2$, the projection quotients down to a map from the nilmanifold
$\HG$ to the torus $\R^2/\Z^2=\T^2$. This map, by abuse of notation, is also
called $P$.

The group isomorphism $f_*$ on $\Gamma$
quotients down to an isomorphism of $\Gamma/Z(\Gamma)\cong \Z^2$.
This new isomorphism
$A:\Z^2\to\Z^2$ can be regarded as a $2\times 2$ matrix.  Further, the
diagram
\[
    \commdiag {\Gamma} {f_*} {\Gamma}
              {P}            {P}
              {\Z^2}   {A}   {\Z^2}
\]
commutes.  In fact, the matrix $A$ is equal to the $2\times 2$ matrix, also
called $A$, studied in Section \ref{nilmanifolds}.  By Proposition
\ref{phmat}, we know this matrix is hyperbolic.  It gives a hyperbolic toral
automorphism $A:\T^2\to\T^2$ and the above diagram can be re-written as
\[
    \commdiag {\pi_1(\HG)}  {f_*} {\pi_1(\HG)}
              {P_*}               {P_*}
              {\pi_1(\T^2)} {A}   {\pi_1(\T^2)}
\]
By the results of J.~Franks \cite{Franks1}, there is a
semi-conjugacy $H:\HG\to\T^2$ such that $H$ and $P$ induce the same action on
the fundamental group,
and such that the diagram
\[
    \commdiag {\HG}  {f_0} {\HG}
              {H}          {H}
              {\T^2} {A}   {\T^2}
\]
commutes.  As stated in \cite{Franks1}, $f_0$ must have a fixed point $x_0$ in
order to find a unique semi-conjugacy such that $H(x_0) = 0$.  If, as is the
case here, uniqueness is not required, we may adapt the proof of Franks to
find a semi-conjugacy $H$, even though $f_0$ does not necessarily have a fixed
point.

Lift $H$ to a map $\H\to\R^2$ which, by more abuse of notation,
will also be called $H$.  After conjugating $f$ with some multiplication, 
assume without loss of generality that $H(0,0,0)=(0,0)$.  Then
$H|_\Gamma=P|_\Gamma$ by the definition of $H$.

Let $g:\H\to\H$ be the algebraic part of $f_0$ as defined in Section
\ref{nilmanifolds}.  $g$ quotients down to $g_0:\HG\to\HG$ and by Proposition
\ref{phmat}, both $g$ and $g_0$ are partially hyperbolic diffeomorphisms.
Moreover, the diagram
\[
    \commdiag {\HG}  {g_0} {\HG}
              {P}          {P}
              {\T^2} {A}   {\T^2}
\]
commutes.  The center bundle of $g$ is given by the vector field
$Z=\deldel{z}$ and the center leaves are vertical lines.  These
quotient down to compact center leaves on the nilmanifold.

Our goal is to produce a leaf conjugacy between $f_0$ and $g_0$, a
homeomorphism $h_0:\HG\to\HG$ which takes center leaves of $g_0$ to center
leaves of $f_0$ and such that
\[ h_0g_0(\LL) = f_0h_0(\LL) \]
for every center leaf $\LL$ of $g_0$.  We achieve this by constructing a leaf
conjugacy $h:\H\to\H$ between $f$ and $g$ on the universal cover such that
\[ h(\gamma\cdot p) = \gamma\cdot h(p) \]
for all $\gamma\in\Gamma$ and $p\in\H$.  $h_0$ is then the quotient of $h$.

The first step is to better understand the semi-conjugacy $H$.

\begin{prop}
    The fibers of $H$ are the center leaves of $f$.  That is, for $p,q\in\H$,
    $H(p)=H(q)$ if and only if $p\in\Wc(q)$.
\end{prop}

\begin{proof}
    Note that as $P$ and $H$ have the same action on the fundamental group,
    the distance between $P(q)$ and $H(q)$ is uniformly bounded for all points
    $q\in\H$.  Then,
    \begin{align*}
        p\in\Wc(q) & \quad \spacedblarrow \\
        \|Pf^n(p) - Pf^n(q)\|
                &\quad \text{is bounded for all $n\in\Z$} \spacedblarrow \\
        \|Hf^n(p) - Hf^n(q)\|
                &\quad \text{is bounded for all $n\in\Z$} \spacedblarrow \\
        \|A^nH(p) - A^nH(q)\|
                &\quad \text{is bounded for all $n\in\Z$}
    \end{align*}
    where the first equivalence is due to the \CL\ and the last equivalence is
    due to the definition of $H$.  As $A$ is a hyperbolic linear map, the last
    condition can be satisfied if and only if $H(p)=H(q)$.
\end{proof}

\begin{prop}
    There is a continuous function $\sigma:\R^2\to\H$ such that $H\circ\sigma$
    is the identity on $\R^2$.
\end{prop}

\begin{proof}
    Fix any point $p_0\in\H$.  The \GPS\ of the invariant
    foliations of $f$ shows that for any $q\in\H$, there are unique points
    $x,y\in\H$ such that
    \[p_0 \stackrel{u}{\leadsto} x
          \stackrel{s}{\leadsto} y
          \stackrel{c}{\leadsto} q . \]
    That is, $x\in\Wu(p_0)\cap\Wcs(q)$ and $y\in\Ws(x)\cap\Wc(q)$.
    Define a map $\omega:\H\to\H$ by $\omega(q)=y$.  The image $\omega(\H)$ is
    a topological surface through the point $p_0$ which intersects every
    center leaf of $f$ exactly once.  In fact, it is a {\em $us$-pseudoleaf}
    as described in \cite{ham-thesis}.  $\omega$ is continuous as the
    stable, center, and unstable foliations of $f$ are continuous and
    transverse.  Further, due to the uniqueness of the intersections $x$ and
    $y$ above, points $q,q'\in\H$ lie on the same center leaf of $f$ if and
    only if $\omega(q)=\omega(q')$.

    As a semi-conjugacy, $H$ is surjective.
    For $v\in\R^2$ take any point $q\in H^{-1}(v)$ and define
    $\sigma:\R^2\to\H$ by $\sigma(v)=\omega(q)$.  This is well-defined as by the
    previous proposition the fibers of $H$ are center leaves, and therefore
    \[
        q,q'\in H^{-1}(v) \spacearrow
        \omega(q) = \omega(q').
    \]
    Moreover, $\omega(q)$ and $q$ lie on the same center leaf by the
    definition of $\omega$, and so $H\sigma(v)=H\omega(q)=H(q)=v$.
    That is, $\sigma$ is a one-sided inverse of $H$ as desired.  All that
    remains is to show that $\sigma$ is continuous.

    For $v\in\R^2$, fix a point a $q\in H^{-1}(v)$ and find a small compact
    topological disk $D$ through $q$ transversal to the center foliation such
    that $v$ lies in the interior of $H(D)$.
    $H|_D$ is continuous and injective and is therefore a homeomorphism onto
    its image.  Then $\sigma|_{H(D)}$ is continuous as the composition of
    $\omega$ and the inverse of $H|_D$.
\end{proof}

With the function $\sigma$ in hand, it is easy to construct leaf conjugacies on
the universal cover $\H$.

\begin{lemma}
    Let $\varphi_t$ be a flow generated by a non-zero
    vector field tangent to $\Ec$.  The homeomorphism
    \begin{equation} \label{simple-h}
        \H\to\H,\quad(x,y,z) \mapsto \varphi_{z} (\sigma(x,y))
    \end{equation}
    defines a leaf conjugacy between $f$ and $g$.
    More generally, for any continuous function,
    $\rho:\R^2\to\R$, the homeomorphism
    \begin{equation} \label{non-simple-h}
        \H\to\H,\quad(x,y,z) \mapsto \varphi_{\rho(x,y)+z} (\sigma(x,y))
    \end{equation}
    defines a leaf conjugacy.
\end{lemma}
\begin{proof}
    Let $h$ be the map defined in \eqref{simple-h}.
    The orbits of $\varphi_t$ are exactly the center leaves of $f$.
    
    Let $\LL$ be a center leaf of $g$.  It is of the form
    $(x,y)\times\R\subset\H$ for real numbers $x$ and $y$.  Then, $g(\LL) =
    A(x,y)\times\R$ and $hg(\LL)$ is the center leaf of $f$ through the point
    $\sigma A(x,y)$.

    On the other hand, $fh(\LL)$ is the center leaf of $f$ through the point
    $f\sigma(x,y)$ and, by the definition of $H$ and of $\sigma$,
    \[
        Hf\sigma(x,y) = AH\sigma(x,y) = A(x,y) = H\sigma A(x,y).
    \]
    Both $fh(\LL)$ and $hg(\LL)$ are the center leaf given by the pre-image
    $H^{-1}A(x,y)$,
    establishing the leaf conjugacy.

    The addition of a function $\rho$ only serves to slide the homeomorphism
    along center leaves.  It does not affect the images of complete center
    leaves and therefore the more general definition given by
    \eqref{non-simple-h} is also a leaf conjugacy.
\end{proof}

This gives a large family of leaf conjugacies on the universal covering space
$\H$.  To find a leaf conjugacy which will quotient down to the nilmanifold
$\HG$ requires a careful choice of both the flow $\varphi_t$ and the offset
function $\rho$.

Define $\varphi_t$ such that the flow is of constant speed along each
individual leaf and that the time-one map takes a point $p\in\H$ to
$(0,0,1)\cdot p$ which, by Proposition \ref{cmptleaf}, lies on the same center
leaf.  More precisely, define a length function $\ell$ by
\[  \ell:\H\to\R,  \quad  \ell(p)=d_c\bigl(p, (0,0,1)\cdot p\bigr)\]
where $d_c$ denotes distance as measured along a center leaf.
This function is continuous due to the continuity of the foliation.  Let $V_1$
be a continuous unit vector field tangent to $E^c$ and consider the re-scaled
vector field
\[ V(p) = \ell(p)V_1(p). \]
If $\varphi_t$ is the flow induced by $V$, then
\begin{align*}
    d_c\bigl(p, \varphi_t(p)\bigr) &= \ell(p)\cdot|t| \spacearrow \\
    d_c\bigl(p, \varphi_1(p)\bigr) &= \ell(p) \spacearrow \\
    \varphi_1(p) &= (0,0,\pm1)\cdot p
\end{align*}
By reversing the flow if necessary, we can ensure that
$\varphi_1(p)=(0,0,1)\cdot p$ as desired.

With $\sigma$ and $\varphi_t$ now fixed, define, for any continuous
$\rho:\R^2\to\R$, the leaf conjugacy
\[ h_\rho:\H\to\H, \quad (x,y,z)\mapsto\varphi_{\rho(x,y)+z}(\sigma(x,y)). \]
The definition of $\varphi_t$ immediately gives the following.

\begin{lemma}
    For any $\rho:\R^2\to\R$ and all $q\in\H$,
    \[ h_\rho\bigl((0,0,1)\cdot q\bigr) = (0,0,1)\cdot h_\rho(q). \]
\end{lemma}

To complete the proof of Theorem \ref{mainthm}, we need only find $\rho$ such
that the above relation holds for all $\gamma\in\Gamma$, not just for
$\gamma=(0,0,1)$.  This special $\rho$ will be constructed piece-by-piece.

\begin{figure}%
\centering
\subfloat[][]{\includegraphics{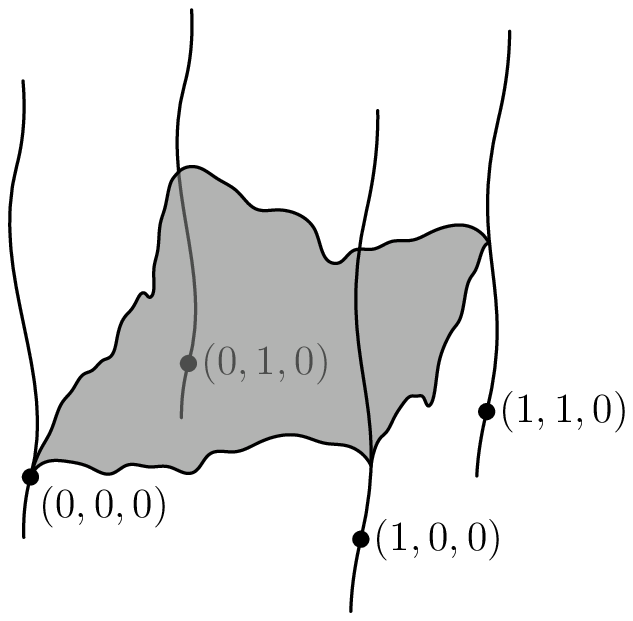}}
\qquad
\subfloat[][]{\includegraphics{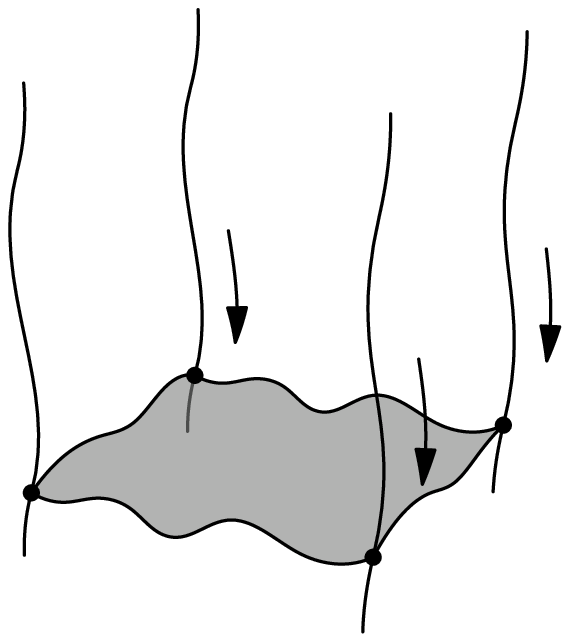}}
\caption{The section $\sigma:\R^2\to\H$ maps the square
$S=[0,1]\times[0,1]\subset\R^2$ to a surface in $\H$, and the four corners of
$S$ to the center leaves of the four points labelled in (a).  The effect of
$\rho$ is to slide the image along center leaves until it meets these four
lattice points, as in (b).
}
%\label{fig:cont}%
\end{figure}
\begin{figure}%
\ContinuedFloat
\centering
\subfloat[][]{\includegraphics{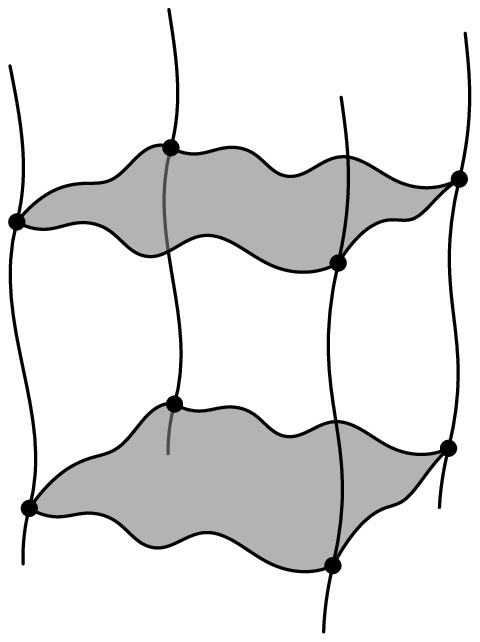}}
\qquad
\subfloat[][]{\includegraphics{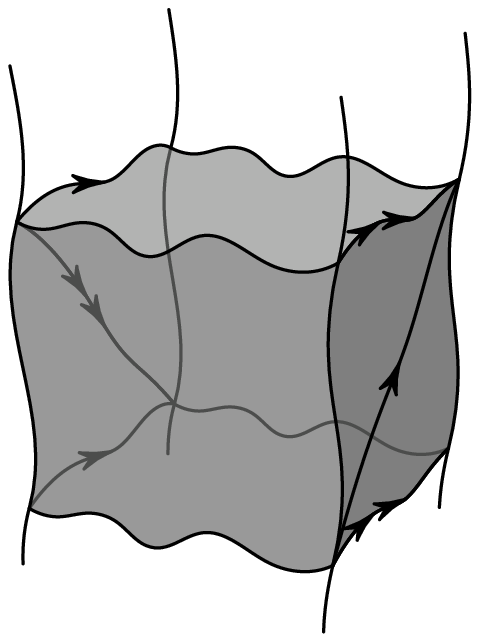}}
\caption[]{ {\small\em (continued)} Applying the deck transformation $p\mapsto (0,0,1)\cdot p$ moves
the surface to one which intersects the same center leaves, (c), and the
region found by taking all of the segments of center leaves between the two
surfaces, (d), gives a fundamental domain for the nilmanifold.
}
%\label{fig:cont}%
\end{figure}

\medskip

First, the assumption that $H(0,0,0)=(0,0)$ implies that there is a unique
value $r\in\R$ such that $\varphi_r(\sigma(0,0))=(0,0,0)$.  Set $\rho(0,0)=r$
so that $h_\rho(0,0,0)=(0,0,0)$.  The semi-conjugacy $H$ was defined to
satisfy the relation
\[ H\bigl((1,0,0)\cdot p\bigr) = H(p) + (1,0). \]
In particular, $H(1,0,0) = (1,0)$ and there is a unique value $\rho(1,0)$ such
that $h_\rho(1,0,0)=(1,0,0)$.  Fix $\rho(1,0)$ as such and then extend $\rho$
to all of the line segment $[0,1]\times\{0\}$ in some continuous way.  The
simplest extension would be 
\[ \rho(t,0) = (1-t)\rho(0,0) + t \rho(1,0). \]

\medskip

With $\rho$ now defined on $[0,1]\times\{0\}\subset\R^2$, $h_\rho$ is defined
on $[0,1]\times\{0\}\times\R\subset\H$ and it is a homeomorphism of that set
onto $H^{-1}([0,1]\times\{0\})$.

Set $\gamma=(0,1,0)\in\Gamma$ and define a homeomorphism
$h:[0,1]\times\{1\}\times\R\to H^{-1}([0,1]\times\{1\})$ by the relation
\[
    h(p) = \gamma\cdot h_\rho(\gamma^{-1}\cdot p).
\]
Note that the flow $\varphi_t$ commutes with multiplication by $\gamma$ and
therefore that
\[ h(x,1,t) = \varphi_t\bigl(h(x,1,0)\bigr) \]
for all $x\in[0,1]$ and $t\in\R$.
Both $h(x,1,0)$ and $\sigma(x,1)$ lie on the center leaf $H^{-1}(x,1)$.  Define
$\rho(x,1)$ as the unique value satisfying
\[ \varphi_{\rho(x,1)}\bigl(\sigma(x,1)\bigr) = h(x,1,0). \]
Then, $h_\rho$ is equal to $h$ on the domain of the latter, and as $h$ is
continuous, $\rho$ is continuous.  Thus, we have extended $\rho$ to the set
$[0,1]\times\{0,1\}\subset\R^2$ and
\[ h_\rho(\gamma\cdot p) = \gamma \cdot h_\rho(p) \]
when $p\in[0,1]\times\{0\}$ and $\gamma=(0,1,0)$.  In particular,
$h_\rho(0,1,0)=(0,1,0)$ and $h_\rho(1,1,0)=(1,1,0)$.

\medskip

Continuously extend $\rho$ so that its domain includes the line segment
$\{0\}\times[0,1]$ (again by linear interpolation, if you like).  Then, using
the same reasoning as above, there is a unique extension of $\rho$ to
$\{1\}\times[0,1]$ such that
\[ h_\rho\bigl((1,0,0)\cdot p\bigr) = (1,0,0)\cdot h(p) \]
for $p\in\{0\}\times[0,1]\times\R\subset\H$.  One must check that this
extension of $\rho$ agrees with the previous definition of $\rho(1,0)$ and
$\rho(1,1)$, but either definition implies that
\[ h_\rho(1,0,z) = (1,0,z) \]
and
\[ h_\rho(1,1,z) = (1,1,z) \]
for all integers $z$.  This uniquely determines the value of $\rho$ at these
two points, so the map $\rho$ is well-defined and continuous on the boundary
$\partial S$ of the square $S=[0,1]\times[0,1]$.  Extend $\rho$ continuously
to the interior of $S$.  Then, $\rho$ is defined on $S$ and $h_\rho$ is
defined on $[0,1]\times[0,1]\times\R\subset\H$.  We now extend this definition
to all of $\H$.

\medskip

Any point $p\in\H$ can be written as
\[ (a,b,0)\cdot(x,y,z) \]
where $(a,b,0)\in\Gamma$ and $(x,y)\in S$.  Using this, define $h:\H\to\H$ as
\[ h\bigl((a,b,0)\cdot(x,y,z)\bigr) = (a,b,0)\cdot h_\rho(x,y,z). \]
The decomposition is unique except for points of the form
\[  (a+1,b,0)\cdot(0,y,z)
        = (a,b,0)\cdot(1,0,0)\cdot(0,y,z)
        = (a,b,0)\cdot(1,y,z+y) \]
or
\[  (a,b+1,a)\cdot(x,0,z)
        = (a,b,0)\cdot(0,1,0)\cdot(x,0,z)
        = (a,b,0)\cdot(x,1,z). \]
The careful definition of $\rho$ on $\partial S$ ensures that in either of the
above cases, the two choices of decomposition give the same value for $h(p)$
and therefore $h$ is well-defined and continuous on all of the Heisenberg
group.

As multiplication by $(a,b,0)\in\Gamma$ commutes with the flow $\varphi_t$,
one has that $\varphi_t\circ h(x,y,z) = h(z,y,z+t)$ so that $h$ is equal to
$h_\rho$ for some $\rho$ now defined on all of $\R^2$.  Thus, $h$ is a
leaf-conjugacy and
\begin{equation} \label{goodh}
    h(\gamma\cdot p) = \gamma\cdot h(p)
\end{equation}
where $\gamma=(0,0,1)$.  Further, for points $(a,b,c)\in\Gamma$ and
$(x,y,z)\in S\times\R$,
\begin{align*}
    h\bigl((a,b,c)\cdot(x,y,z)\bigr)
        &= h\bigl((a,b,0)\cdot(0,0,c)\cdot(x,y,z)\bigr) \\
        &= (a,b,0)\cdot h_\rho\bigl((0,0,c)\cdot(x,y,z)\bigr) \\
        &= (a,b,0)\cdot(0,0,c)\cdot h_\rho(x,y,z) \\
        &= (a,b,c)\cdot h(x,y,z).
\end{align*}
For any lattice point $\gamma\in\Gamma$ and any point $p\in\H$, let
$p=\eta\cdot q$ be such that $\eta\in\Gamma$ and $q\in S\times\R$.
Then,
\[
    h(\gamma\cdot p)
        = h(\gamma\cdot\eta\cdot q)
        = \gamma\cdot\eta\cdot h(q)
        = \gamma\cdot h(\eta\cdot q)
        = \gamma\cdot h(p).
\]
This shows that the leaf conjugacy $h$ on the universal cover $\H$ descends to
the original nilmanifold $\HG$.  It is a conjugacy between the partially
hyperbolic system $f_0$ and the nilmanifold automorphism $g_0$ and the proof
of Theorem \ref{mainthm} is complete.

% Epilogue {{{1

\section*{Acknowledgments} The author would like to thank Charles Pugh, 
Enrique Pujals, and Ra\'ul Ures for helpful discussions during the preparation
of this paper.  This research was partly funded by CNPq of Brazil.

\bibliographystyle{plain}
\bibliography{dynamics}

\end{document}